\documentclass[a4paper,10pt]{article}

\usepackage{mathrsfs}
\usepackage{amsmath}
\usepackage{bm}
\usepackage{amssymb}
\usepackage{makeidx}
\makeindex
\usepackage[dvipsnames]{xcolor}
\usepackage[colorlinks=true, linkcolor=blue, citecolor=blue, urlcolor=blue]{hyperref}
\usepackage{titlesec}
\usepackage{amsthm}
\usepackage{setspace}
\usepackage{pifont}
\usepackage{esint}
\usepackage{graphicx}
\usepackage{float}
\usepackage{enumerate}
\numberwithin{equation}{section}
\usepackage[backend=biber,style=alphabetic,sorting=nyt]{biblatex}
\DeclareFieldFormat[article]{title}{#1}
\renewbibmacro{in:}{}
\DeclareFieldFormat{pages}{#1}
\DeclareFieldFormat[article]
{journaltitle}{\mkbibemph{#1}\addcomma}
\renewbibmacro*{volume+number+eid}{%
	\printfield{volume}%
	\setunit{\addspace}%
	\printfield{eid}%
}
\renewbibmacro*{eprint}{%
	\iffieldundef{eprint}
	{}
	{%
		\setunit{\addcomma\space}
		\printtext{%
			\printfield[eprint:arxiv]{eprint}}%
	}%
}

\titleformat{\section}[block]{\bfseries\Large}{\thesection}{1em}{}

\newtheoremstyle{mystyle}
{}{}{\upshape}{}{\bfseries}{.}{.5em}{}
\theoremstyle{mystyle}
\newtheorem{theorem}{Theorem}[section]

\newtheorem{proposition}[theorem]{Proposition}
\newtheorem{coro}[theorem]{Corollary}
\newtheorem{lemma}[theorem]{Lemma}

\newtheorem{remark}{Remark}[section]
\theoremstyle{plain}

\begin{document} 
	
\title{\large
	Large-Time Asymptotics for Heat and Fractional Heat Equations\\ on the Lattice and General Finite Subgraphs}

\author{\small Rui Chen,\qquad Bo Li }
\date{}
\maketitle
\thispagestyle{empty}
\pagenumbering{arabic}

\noindent\textbf{Abstract:}
In this paper, we study large-time asymptotics for heat and fractional heat equations in two discrete settings: the full lattice \(\mathbb Z^d\) and finite connected subgraphs with Dirichlet boundary condition. These results provide a unified discrete theory of long-time asymptotics for local and nonlocal diffusions.

For \(d\ge1\) and \(s\in(0,1]\), we consider on \(\mathbb Z^d\) the Cauchy problem
\[
\partial_t u+(-\Delta)^s u=0,\qquad u(0)=u_0\in \ell^1(\mathbb Z^d),
\]
and derive a precise first-order asymptotic expansion toward the lattice fractional heat kernel \(G_t^{(s)}\). The main technical input is a pair of sharp translation-increment bounds for \(G_t^{(s)}\): a pointwise estimate and an \(\ell^1\)-estimate. As consequences, under finite first moment we obtain the optimal decay rate \(t^{-1/(2s)}\) in \(\ell^p\)-asymptotics (\(1\le p\le\infty\)), and we prove sharpness by explicit shifted-kernel examples. Without moment assumptions, we still establish convergence in the full \(\ell^1\)-class, and we show that no universal quantitative rate can hold in general.

We also analyze fractional Dirichlet diffusion on finite connected subgraphs (restricted fractional setting, including \(s=1\) as the local case). In this finite-dimensional framework, solutions admit spectral decomposition and exhibit exponential large-time behavior governed by the principal eigenvalue and the spectral gap. In addition, we study positivity improving properties of the associated semigroups for both the lattice and Dirichlet evolutions.



\bigskip

\tableofcontents
\thispagestyle{empty}

\setcounter{page}{1}
\section{Introduction and Main Results}

The aim of this paper is to investigate the large-time asymptotic behavior of heat and fractional heat equations in two related discrete settings: the whole lattice \(\mathbb Z^d\) and finite connected subgraphs with Dirichlet exterior condition.

Let \(d\ge 1\) and \(s\in(0,1]\).  
On  \(\mathbb Z^d\), for initial data \(u_0\in \ell^1(\mathbb Z^d)\), we consider the Cauchy problem
\begin{equation}\label{eq:cauchy-frac-lap}
\begin{cases}
\partial_t u(t,x)+(-\Delta)^s u(t,x)=0, & t>0,\ x\in\mathbb Z^d,\\[1mm]
u(0,x)=u_0(x), & x\in\mathbb Z^d.
\end{cases}
\end{equation}
Here \((-\Delta)^s\) (\(0<s\le1\)) denotes the fractional power  of the discrete Laplacian on \(\mathbb Z^d\), defined via Fourier multiplier
\[
\widehat{(-\Delta)^s f}(\xi)=\omega(\xi)^s\widehat f(\xi),
\qquad
\omega(\xi):=4\sum_{j=1}^d\sin^2\!\Big(\frac{\xi_j}{2}\Big),
\quad \xi\in\mathbb T^d:=[-\pi,\pi]^d.
\]
The Cauchy problem (\ref{eq:cauchy-frac-lap}) admits a unique solution in \(\ell^1(\mathbb Z^d)\), given by (see Section \ref{Foundations of Discrete Operators})
\[
u(t,x)=G_t^{(s)}*u_0,\qquad t\ge0,
\]
where
\begin{equation}\label{gaussian}
    G_t^{(s)}(x)
=
\frac{1}{(2\pi)^{d}}
\int_{\mathbb T^d}
e^{-t\omega(\xi)^s}e^{i\langle x,\xi\rangle}\,d\xi,
\qquad x\in\mathbb Z^d.
\end{equation}

In parallel, let \(G=(V,E,\mu,w)\) be a locally finite weighted connected graph, and let \(\Omega\subset V\) be a finite connected subset.  
For \(s\in(0,1]\), we denote by \(L_{\Omega,s}^D\) the Dirichlet diffusion operator on \(\Omega\): for \(0<s<1\), \(L_{\Omega,s}^D\) is the restricted Dirichlet fractional Laplacian (zero extension outside \(\Omega\), then restriction back to \(\Omega\)); for \(s=1\), \(L_{\Omega,1}^D\) is the usual Dirichlet graph Laplacian, see Section \ref{Setting and Equation}. We study
\begin{equation}\label{eq:finite-dirichlet-frac-heat-restricted}
\begin{cases}
\partial_t u(t,x)+L_{\Omega,s}^D u(t,x)=0,& t>0,\ x\in\Omega,\\
u(t,x)=0,& t>0,\ x\in V\setminus\Omega,\\
u(0,x)=u_0(x),& x\in\Omega.
\end{cases}
\end{equation}

In the following, we will show that the two models exhibit different asymptotic mechanisms. On \(\mathbb Z^d\), decay is polynomial and governed by the lattice fractional heat kernel, with sharp first-order correction controlled by the first moment of \(u_0\). On finite subgraphs, the spectrum is discrete and the long-time profile is the first eigenvalue, with exponential convergence rate given by the spectral gap.

It is well known that diffusion phenomena are ubiquitous in natural sciences, and parabolic equations provide a fundamental mathematical framework to describe their evolution. 
From the PDE viewpoint, one of the central questions is the \emph{large-time asymptotic behavior}: after long time, what profile dominates the solution, and at which rate does the convergence occur?

In the classical case (\(s=1\)), the heat equation in \(\mathbb R^d\),
\[
\partial_t u-\Delta u=0,\qquad u(0,\cdot)=u_0\in L^1(\mathbb R^d),
\]
admits the convolution representation 
\[
u(t,x)=\big(G_t^{(1)}*u_0\big)(x),\qquad 
G_t^{(1)}(x)=(4\pi t)^{-d/2}e^{-|x|^2/(4t)}.
\]
Observe that integrating over all of $\mathbb R^d$, we obtain that the total mass of solutions is conserved for all time, that is,
\[
M:=\int_{\mathbb R^d}u_0(x)\,dx=\int_{\mathbb R^d}u(t,x)\,dx
\]
for all \(t>0\), and \(u(t,\cdot)\) converges to \(M G_t^{(1)}\) in the classical self-similar scaling sense; equivalently, for every \(1\le p\le\infty\),
\[
t^{\frac d2\left(1-\frac1p\right)}
\|u(t)-M G_t^{(1)}\|_{L^p(\mathbb R^d)}\longrightarrow 0,
\qquad t\to\infty.
\]
This is a standard cornerstone in the asymptotic theory of the heat equation; see, e.g., \cite{vazquez2017asymptotic}.

A parallel theory holds for the fractional heat equation (\(0<s<1\)):
\[
\partial_t u+(-\Delta)^s u=0,\qquad u(0,\cdot)=u_0\in L^1(\mathbb R^d),
\]
whose solution is
\[
u(t,x)=\big(P_t^{(s)}*u_0\big)(x),
\]
where \(P_t^{(s)}\) is the fractional heat kernel (stable density), characterized by $\widehat{P_t^{(s)}}(\xi)=e^{-t|\xi|^{2s}}.$
Again, mass is conserved, and the large-time profile is \(M P_t^{(s)}\). 
More precisely, for all \(1\le p\le\infty\),
\[
t^{\frac d{2s}\left(1-\frac1p\right)}
\|u(t)-M P_t^{(s)}\|_{L^p(\mathbb R^d)}\to0.
\]
See, for instance, \cite{vazquez2018asymptotic}.

A finer first-order asymptotic expansion is available under finite first absolute moment:
\[
\mathcal N_1(u_0):=\int_{\mathbb R^d}|x|\,|u_0(x)|\,dx<\infty.
\]
In that case one has the quantitative estimates
\[
t^{d/(2s)}\|u(t)-M P_t^{(s)}\|_{L^\infty(\mathbb R^d)}
\le C\,\mathcal N_1(u_0)\,t^{-1/(2s)}, \quad s\in (0,1]
\]
and
\[
\|u(t)-M P_t^{(s)}\|_{L^1(\mathbb R^d)}
\le C\,\mathcal N_1(u_0)\,t^{-1/(2s)},\quad s\in (0,1]
\]
with corresponding \(L^p\)-versions by interpolation. Moreover, the rate \(t^{-1/(2s)}\) is optimal within this first-moment class; see \cite{vazquez2017asymptotic,vazquez2018asymptotic}.

In bounded domains with homogeneous Dirichlet condition, the long-time behavior is no longer self-similar as in $\mathbb R^d$. 
Instead, it is governed by the spectral structure of the Dirichlet operator: the first mode dominates as $t\to\infty$, and the remainder decays at an exponential rate determined by the spectral gap $\lambda_2-\lambda_1$. 
Hence the convergence to the principal profile is exponential, in sharp contrast with the algebraic decay rates in the whole-space setting. See, e.g., \cite{chasseigne2006asymptotic,bonforte2014existence}.

The Euclidean large-time theory has natural counterparts on curved spaces, where geometry can substantially modify both asymptotic profiles and decay mechanisms.
A classical model case is $\mathbb H^n$, where the asymptotics differ from the Euclidean self-similar regime due to geometric effects \cite{vazquez2022asymptotic}.

On complete noncompact manifolds, large-time asymptotics has been analyzed under suitable Ricci curvature assumptions, highlighting geometry as a decisive factor in the long-time regime \cite{li1986large,grigor2023asymptotic}; for broader classes, notably noncompact Riemannian symmetric spaces, representation-theoretic methods yield sharper results for both heat and fractional heat equations, including quantitative \(L^p\)-asymptotic estimates \cite{anker2023asymptotic,papageorgiou2024asymptotic,naik2024p,papageorgiou2024large}.

Compared with the Euclidean and manifold settings, large-time asymptotics for diffusion equations on graphs is still relatively less developed, especially for lattice graphs and nonlocal operators. 
Related analytical tools for the discrete Laplacian can be found in \cite{ciaurri2023harmonic}. 
For nonlocal discrete models, including the fractional discrete Laplacian together with regularity and applications, see \cite{ciaurri2018nonlocal}. For discrete evolutions, asymptotic questions have been addressed, for instance, in the study of the discrete-in-time heat equation \cite{abadias2022asymptotic}. 
On metric graphs, several works analyze diffusion asymptotics for both local and nonlocal models, including general diffusion problems and configurations with infinite edges, where effective reduced dynamics may emerge at large time \cite{ignat2021asymptotic}. 
For the one-dimensional lattice $\mathbb Z$, refined large-time behavior for the discrete heat equation (including moment methods and decay rates) has also been obtained \cite{abadias2021large}. 
Nevertheless, a unified asymptotic theory for discrete Laplacian and fractional Laplacian flows on higher-dimensional lattices, as well as on finite subgraphs with Dirichlet boundary conditions, remains largely open.

 These previous asymptotic results in different background provide the starting point for our discrete framework. 
In this paper, we derive their counterparts on \(\mathbb Z^d\), obtain sharp first-order decay rates together with optimality, and then highlight the qualitative transition to finite subgraphs with Dirichlet boundary conditions, where the large-time behavior is governed by exponential spectral decay.

Our first step is to establish sharp translation-increment bounds for the lattice fractional heat kernel \(G_t^{(s)}\), which are the main technical inputs for the first-order asymptotic analysis. 
Unlike the Euclidean stable kernel, \(G_t^{(s)}\) on \(\mathbb Z^d\) does not satisfy an exact self-similar scaling law, which makes the proof substantially more delicate. 
Our strategy is to recast the problem into the analysis of the associated Fourier multipliers on \(\mathbb T^d\).

Here and throughout, for \(x=(x_1,\dots,x_d)\in\mathbb Z^d\), we set
\[
|x|:=\bigl(x_1^2+\cdots+x_d^2\bigr)^{1/2}.
\]

\begin{proposition}
\label{prop:kernel-increment-pointwise}
Let \(d\ge1\), \(s\in(0,1]\). Then there exists \(C=C(d,s)>0\) such that for all \(t>0\), \(x,y\in\mathbb Z^d\),
\begin{equation}
\label{eq:linfty-difference-kernel-sharp}
\bigl|G_t^{(s)}(x-y)-G_t^{(s)}(x)\bigr|
\le
C\,t^{-\frac d{2s}}
\min\!\left\{1,\ |y|\,t^{-\frac1{2s}}\right\},
\end{equation}
where $G_t^{(s)}$ is defined in (\ref{gaussian}).
\end{proposition}

The proof proceeds by reducing 
\(\bigl|G_t^{(s)}(x-y)-G_t^{(s)}(x)\bigr|\) 
to an estimate of a Fourier integral over \(\mathbb T^d\), from which \eqref{eq:linfty-difference-kernel-sharp} follows.

\begin{proposition}
\label{prop:l1-sharp-by-splitting}
Let \(d\ge1\), \(s\in(0,1]\). Then there exists \(C=C(d,s)>0\) such that for all \(t>0\) and \(y\in\mathbb Z^d\),
\begin{equation}
\label{eq:l1-sharp-main}
\sum_{x\in\mathbb Z^d}\bigl|G_t^{(s)}(x-y)-G_t^{(s)}(x)\bigr|
\le
C\,\min\!\left\{1,\ |y|\,t^{-1/(2s)}\right\},
\end{equation}
where $G_t^{(s)}$ is defined in (\ref{gaussian}).
\end{proposition}

Proposition~\ref{prop:l1-sharp-by-splitting} is one of the technically most delicate parts of the paper.
Obtaining a \emph{sharp} \(\ell^1\)-estimate is difficult because one must control oscillatory Fourier
integrals on the torus while keeping the exact scaling
\(
\min\{1,\ |y|\,t^{-1/(2s)}\}.
\)
A rough bound loses either the \(y\)-dependence or the optimal time decay.

To preserve sharpness, we decompose the increment multiplier $(e^{i y\cdot \xi}-1)e^{-t\omega(\xi)^s}$
into a low-frequency and a high-frequency part.  
The low-frequency contribution captures the main scaling structure:
after rescaling, it behaves like a stable profile and yields the critical factor
\(
|y|\,t^{-1/(2s)}.
\)
The high-frequency contribution is controlled by the strong damping of
\(e^{-t\omega(\xi)^s}\) away from \(\xi=0\), together with Sobolev--Fourier estimates on
\(\mathbb T^d\) (see Lemma \ref{lem:Wm1-to-l1}).
A careful treatment of the transition region between these two regimes is essential to avoid derivative
loss and to close the estimate at the optimal order.
Combining both parts gives \eqref{eq:l1-sharp-main}.

With the above two propositions at hand, we derive the first of our main theorems.

\begin{theorem}
\label{thm:lattice-first-order}
Assume \(u_0\in \ell^1(\mathbb Z^d)\) and \(\mathcal N_1(u_0)<\infty\). Let $u(t) = u(t,\cdot)$ be the solution of (\ref{eq:cauchy-frac-lap}). Then there exists \(C=C(d,s)>0\) such that for all \(t>0\),
\begin{equation}
\label{eq:lattice-Linfty-rate}
t^{\frac d{2s}}\,
\|u(t)-M G_t^{(s)}\|_{\ell^\infty(\mathbb Z^d)}
\le
C\,\mathcal N_1(u_0)\,t^{-\frac1{2s}},
\end{equation}
and
\begin{equation}
\label{eq:lattice-L1-rate}
\|u(t)-M G_t^{(s)}\|_{\ell^1(\mathbb Z^d)}
\le
C\,\mathcal N_1(u_0)\,t^{-\frac1{2s}}.
\end{equation}
Consequently, for every \(1\le p\le\infty\),
\begin{equation}
\label{eq:lattice-Lp-rate}
t^{\frac d{2s}\left(1-\frac1p\right)}
\|u(t)-M G_t^{(s)}\|_{\ell^p(\mathbb Z^d)}
\le
C\,\mathcal N_1(u_0)\,t^{-\frac1{2s}}.
\end{equation}
\end{theorem}

Unlike the Euclidean fractional heat kernel, the lattice kernel is not exactly self-similar; nevertheless, it satisfies a limit scaling toward the stable profile in self-similar variables. Define
\[
\Phi_s(\eta):=
\frac1{(2\pi)^d}\int_{\mathbb R^d}
e^{-|\zeta|^{2s}}e^{i\eta\cdot\zeta}\,d\zeta,
\qquad \eta\in\mathbb R^d.
\]

\begin{proposition}
\label{lem:lattice-stable-scaling}
Let \(d\ge1\), \(s\in(0,1]\). 
For each \(t>0\), let \(x_t:\mathbb R^d\to\mathbb Z^d\) be any selection satisfying
\begin{equation}
\label{eq:nearest-lattice-selection}
\sup_{\eta\in\mathbb R^d}\big|x_t(\eta)-t^{1/(2s)}\eta\big|\le C_0
\end{equation}
for some constant \(C_0>0\) independent of \(t\).
Then, for every \(R>0\),
\begin{equation}
\label{eq:lclt-uniform-compact}
\sup_{|\eta|\le R}
\left|
t^{d/(2s)}\,G_t^{(s)}\!\big(x_t(\eta)\big)-\Phi_s(\eta)
\right|
\rightarrow 0,
\quad t\to\infty.
\end{equation}
In particular, if \(x=t^{1/(2s)}\eta+O(1)\) with \(|\eta|\le R\), then
\[
G_t^{(s)}(x)=t^{-d/(2s)}\Phi_s(\eta)+o\,\big(t^{-d/(2s)}\big),
\]
locally uniformly in \(\eta\).
\end{proposition}

Proposition~\ref{lem:lattice-stable-scaling} provides a lattice analogue of the self-similar scaling
for the Euclidean \(2s\)-stable kernel.
This asymptotic identification is the key tool for transferring sharp lower bounds from the
continuum to the lattice. In particular, it allows us to quantify the effect of a microscopic
translation \(x\mapsto x-e_1\) at large times: on the self-similar scale, such a shift corresponds
to a perturbation of size \(t^{-1/(2s)}\), so the leading-order difference is governed by the
first derivative \(\partial_1\Phi_s\). The next proposition exploits this observation to show
that the decay rate \(t^{-1/(2s)}\) in our first-order asymptotics cannot be improved.

\begin{proposition}
\label{prop:optimality-first-order}
Let \(d\ge1\), \(s\in(0,1]\), and \(1\le p\le\infty\).
There exist \(c=c(d,s,p)>0\) and \(t_0\ge1\) such that for all \(t\ge t_0\),
\[
t^{\frac d{2s}\left(1-\frac1p\right)}
\big\|G_t^{(s)}(\cdot-e_1)-G_t^{(s)}\big\|_{\ell^p(\mathbb Z^d)}
\ge c\,t^{-\frac1{2s}}.
\]
Hence the rate \(t^{-1/(2s)}\) in Theorem~\ref{thm:lattice-first-order} is sharp.
\end{proposition}

We now turn to the large-time asymptotics for general integrable initial data, without any
moment assumption. In contrast to the finite first-moment regime, one cannot expect a
quantitative convergence rate in general (see Proposition~\ref{prop:no-universal-rate-lattice}
below). Nevertheless, the mass conservation and the regularizing effect of the semigroup
still force the solution to converge, after the natural parabolic normalization, to the
fundamental solution profile \(M G_t^{(s)}\). The following theorem records this qualitative
\(\ell^1\) and \(\ell^\infty\) convergence, and by interpolation yields convergence in all
\(\ell^p\)-norms.

\begin{theorem}
\label{thm:lattice-asymptotic-no-moment}
Assume \(u_0\in \ell^1(\mathbb Z^d)\).  Let $u(t) = u(t,\cdot)$ be the solution of (\ref{eq:cauchy-frac-lap}). Then
\begin{equation}
\label{eq:lattice-L1-conv-no-moment}
\lim_{t\to\infty}\|u(t)-M G_t^{(s)}\|_{\ell^1(\mathbb Z^d)}=0,
\end{equation}
and
\begin{equation}
\label{eq:lattice-Linfty-conv-no-moment}
\lim_{t\to\infty}t^{\frac d{2s}}
\|u(t)-M G_t^{(s)}\|_{\ell^\infty(\mathbb Z^d)}=0.
\end{equation}
Consequently, for every \(1\le p\le\infty\),
\begin{equation}
\label{eq:lattice-Lp-conv-no-moment}
\lim_{t\to\infty}
t^{\frac d{2s}\left(1-\frac1p\right)}
\|u(t)-M G_t^{(s)}\|_{\ell^p(\mathbb Z^d)}=0.
\end{equation}
\end{theorem}

The first-order asymptotics obtained under a finite first-moment assumption are quantitative and
sharp, but this is no longer the case for general integrable data. In fact, even after the natural
self-similar normalization \(t^{d/(2s)}\), one cannot prescribe a rate of convergence that holds
uniformly over all unit-mass nonnegative initial data in \(\ell^1(\mathbb Z^d)\).
The next proposition provides a counterexample mechanism: given any prescribed decay function
\(\phi(t)\downarrow0\), one can construct an \(\ell^1\) datum whose convergence along a sequence of
times is at least as slow as \(k\,\phi(t_k)\). As an immediate consequence, there is no uniform
estimate of the form \eqref{eq:no-uniform-rate-unit-mass} valid for all such data.

\begin{proposition}
\label{prop:no-universal-rate-lattice}
Let \(d\ge1\), \(s\in(0,1]\), and let \(\phi:[1,\infty)\to(0,\infty)\) be any decreasing function with
\[
\phi(t)\xrightarrow[]{t\to\infty}0.
\]
Then there exist \(u_0\in \ell^1(\mathbb Z^d)\), \(u_0\ge0\), \(\sum_{x\in\mathbb Z^d}u_0(x)=1\), and a sequence
\(t_k\to\infty\) such that, for the solution
\[
u(t,\cdot)=G_t^{(s)}*u_0,
\]
one has
\begin{equation}
\label{eq:no-universal-rate-lattice}
t_k^{\frac d{2s}}\,
\|u(t_k)-G_{t_k}^{(s)}\|_{\ell^\infty(\mathbb Z^d)}
\ge
k\,\phi(t_k),\qquad k\in\mathbb N.
\end{equation}
\end{proposition}

\begin{coro}
\label{cor:no-uniform-rate-unit-mass}
Let \(d\ge1\), \(s\in(0,1]\), and let \(\psi:[1,\infty)\to(0,\infty)\) satisfy
\(\psi(t)\to0\) as \(t\to\infty\).
Then there is no constant \(C>0\) such that
\begin{equation}
\label{eq:no-uniform-rate-unit-mass}
t^{\frac d{2s}}
\|S_s(t)u_0-G_t^{(s)}\|_{\ell^\infty(\mathbb Z^d)}
\le
C\,\psi(t),
\quad
\forall t\ge1,\ \forall u_0\in\ell^1(\mathbb Z^d),\ u_0\ge0,\ \sum_{x\in \mathbb{Z}^d}u_0(x)=1.
\end{equation}
\end{coro}

Next, we turn to the finite subgraphs setting. 
On a finite connected domain \(\Omega\subset V\) with Dirichlet condition outside \(\Omega\), the generator
\(L_{\Omega,s}^D\) is a finite-dimensional self-adjoint positive operator on \(\ell^2(\Omega,m)\), and the
associated semigroup has a purely discrete spectrum. As a consequence, the long-time behavior is not
self-similar (polynomial) as on \(\mathbb Z^d\), but rather governed by the spectral gap and the first
eigenvalue: after renormalization by \(e^{\mu_{1,s}t}\), the solution converges exponentially fast to a
multiple of the ground state \(\psi_{1,s}\). The next theorem records this first-mode asymptotics in all
\(\ell^p(\Omega,m)\) norms.

\begin{theorem}
\label{thm:finite-subgraph-first-mode}
Let \(\Omega\subset V\) be finite and connected, and let \(s\in(0,1]\).
Let \(u\) be the solution to \eqref{eq:finite-dirichlet-frac-heat-restricted} with
\(u_0\in \ell^2(\Omega,m)\). Then, for every \(1\le p\le\infty\), there exists
\(C_{p,\Omega}>0\) such that for all \(t\ge0\),
\begin{equation}
\label{eq:finite-subgraph-main-est}
\left\|
u(t)-e^{-\mu_{1,s} t}\,\langle u_0,\psi_{1,s}\rangle_{\ell^2(\Omega,m)}\,\psi_{1,s}
\right\|_{\ell^p(\Omega,m)}
\le
C_{p,\Omega}\,e^{-\mu_{2,s} t}\,\|u_0\|_{\ell^2(\Omega,m)}.
\end{equation}
Equivalently,
\begin{equation}
\label{eq:finite-subgraph-renormalized-limit}
\left\|
e^{\mu_{1,s} t}u(t)-\langle u_0,\psi_{1,s}\rangle_{\ell^2(\Omega,m)}\,\psi_{1,s}
\right\|_{\ell^p(\Omega,m)}
\le
C_{p,\Omega}\,e^{-(\mu_{2,s}-\mu_{1,s})t}\,\|u_0\|_{\ell^2(\Omega,m)}.
\end{equation}
In particular,
\[
e^{\mu_{1,s} t}u(t)\longrightarrow
\langle u_0,\psi_{1,s}\rangle_{\ell^2(\Omega,m)}\,\psi_{1,s}
\quad\text{in }\ell^p(\Omega,m),\qquad t\to\infty.
\]
\end{theorem}

Finally, we record a qualitative feature shared by both dynamics considered in this paper:
strict positivity propagates instantly.  On the full lattice \(\mathbb Z^d\), the fractional heat
semigroup \(S_s(t)=e^{-t(-\Delta)^s}\) is a convolution operator with kernel \(G_t^{(s)}\), and the
underlying jump process makes the heat kernel strictly positive at every lattice site for any
positive time.  On a finite connected domain with Dirichlet condition, the restricted
Dirichlet operator \(L_{\Omega,s}^D\) generates a finite-dimensional positive semigroup
\(S_{\Omega,s}(t)=e^{-tL_{\Omega,s}^D}\); irreducibility of the associated generator implies that the
matrix \(S_{\Omega,s}(t)\) has strictly positive entries for every \(t>0\).  In both settings, any
nontrivial nonnegative initial datum yields a solution that becomes strictly positive everywhere
for all \(t>0\).

\begin{proposition}
\label{prop:zd-positivity-improving}
\(S_s(t)\) is positivity improving on \(\ell^p(\mathbb Z^d)\), \(1\le p\le\infty\):
\[
u_0\ge0,\ u_0\not\equiv0
\ \Longrightarrow\
(S_s(t)u_0)(x)>0,\ \forall x\in\mathbb Z^d,\ \forall t>0.
\]
In particular, if \(u_0\ge0\), \(u_0\not\equiv0\), then the solution to
\eqref{eq:cauchy-frac-lap} satisfies
\[
u(t,x)>0,\quad \forall t>0,\ \forall x\in\mathbb Z^d.
\]
\end{proposition}

\begin{proposition}
\label{prop:finite-dirichlet-positivity-improving}
For every \(t>0\), \(S_{\Omega,s}(t)\) is positivity improving:
\[
u_0\ge0,\ u_0\not\equiv0
\ \Longrightarrow\
\big(S_{\Omega,s}(t)u_0\big)(x)>0,\ \forall x\in\Omega.
\]
In particular, if \(u_0\ge0\), \(u_0\not\equiv0\), then the solution to
\eqref{eq:finite-dirichlet-frac-heat-restricted} satisfies
\[
u(t,x)>0,\quad \forall t>0,\ \forall x\in\mathbb Z^d.
\]
\end{proposition}

\smallskip

    The remainder of this paper is organized as follows. In Section~2, we first introduce the discrete Laplacian and the fractional Laplacian on \(\mathbb Z^d\), together with the corresponding Fourier representation and semigroup framework. We then establish the two key kernel increment estimates, namely the sharp pointwise bound and the sharp \(\ell^1\)-bound, which form the technical backbone of the large-time analysis. Building on these estimates, we prove Theorem~\ref{thm:lattice-first-order}  and Theorem~\ref{thm:lattice-asymptotic-no-moment}. We also show that the rate \(t^{-1/(2s)}\) in Theorem~\ref{thm:lattice-first-order} is optimal, and we discuss the non-existence of a universal quantitative rate in the whole \(\ell^1\) class without moment assumptions. In Section~3, we study the Dirichlet problem on finite subgraphs: we define the Dirichlet Laplacian and the corresponding fractional Dirichlet operators, and then prove the large-time asymptotic behavior of solutions. Finally, in Section~4, we investigate positivity improving properties of the associated semigroups, for both the full lattice evolution and the Dirichlet evolution on finite subgraphs.

\section{Fractional Diffusion on the Lattice}

In this section, we study fractional diffusion on the discrete spatial domain
\(\mathbb Z^d\). We begin by introducing the lattice Laplacian and its fractional
powers, together with their Fourier representation on \(\mathbb T^d\) and the
associated semigroup formulation of the evolution problem. These structural tools
provide the analytic framework for the large-time analysis carried out in the rest
of the section.

We then investigate the asymptotic behavior of solutions as \(t\to\infty\).
First, we establish pointwise and \(\ell^1\)-bounds for kernel increments, which
yield the key \(\ell^\infty\)- and \(\ell^1\)-type controls for the error
\(u(t)-M G_t^{(s)}\). Based on these estimates, we derive the optimal convergence
rate in the class of initial data with finite first moment and prove the
corresponding sharpness result. Finally, by an approximation argument, we extend
the asymptotic convergence theory from the finite-first-moment setting to the
general \(\ell^1\)-class, obtaining convergence without a universal quantitative
rate in full generality.

	\subsection{Foundations of Discrete Operators}\label{Foundations of Discrete Operators}

In this section, the spatial domain is the integer lattice
\[
\mathbb{Z}^d=\{x=(x_1,\dots,x_d):x_j\in\mathbb{Z}\}.
\]
We denote by $\ell^p(\mathbb{Z}^d)$, $1\le p\le\infty$, the usual sequence spaces with norms
\[
\|u\|_{\ell^p}=
\begin{cases}
\left(\sum_{x\in\mathbb{Z}^d}|u(x)|^p\right)^{1/p}, & 1\le p<\infty,\\[1mm]
\sup_{x\in\mathbb{Z}^d}|u(x)|, & p=\infty.
\end{cases}
\]
The counting measure is always understood on $\mathbb{Z}^d$. The dual (frequency) domain of $\mathbb{Z}^d$ is the $d$-torus
\[
\mathbb{T}^d=[-\pi,\pi]^d.
\]

For $x\in\mathbb{Z}^d$, we write $y\sim x$ if $y=x\pm e_j$ for some $j\in\{1,\dots,d\}$.
The discrete Laplacian on $\mathbb{Z}^d$ is defined by
\[
-\Delta u(x):=\sum_{y\sim x}\bigl(u(x)-u(y)\bigr),\qquad x\in\mathbb{Z}^d.
\]
Equivalently,
\[
\Delta u(x)=\sum_{j=1}^d\bigl(u(x+e_j)+u(x-e_j)-2u(x)\bigr).
\]

The Fractional Laplacian $(-\Delta)^s$ for $0<s<1$ is defined via the Bochner integral formula (see \cite{chen2025logarithmic}):
\[(-\Delta)^{s}u=\frac{s}{\Gamma(1-s)}
	\int_{0}^{\infty}\bigl(u-e^{t\Delta}u\bigr)\,t^{-1-s}\,dt, \quad u\in \ell^{\infty}(\mathbb{Z}^d).\]
It is straightforward to verify that $\left(-\Delta\right)^{s}u \in \ell^1(\mathbb{Z}^d)$  for every $u\in\ell^1(\mathbb{Z}^d)$. By \cite[Proposition 1.6]{chen2025logarithmic}, \((-\Delta)^s u\to -\Delta u\) in \(\ell^\infty(\mathbb{Z}^d)\) as \(s\to1^-\) for every \(u\in\ell^\infty(\mathbb{Z}^d)\).

The following proposition shows that, for every $1\le p\le\infty$, both $-\Delta$ and $(-\Delta)^s$ ($0<s<1$) are bounded linear operators on $\ell^p(\mathbb Z^d)$.

\begin{proposition}\label{prop:lp-boundedness-lap-frac}
For every \(1\le p\le \infty\) and \(0<s<1\), the Laplacian $-\Delta$ and the fractional Laplacian $(-\Delta)^s$ both define bounded linear maps on \(\ell^p(\mathbb Z^d)\). More precisely:
\[
\|-\Delta u\|_{\ell^p}\le 4d\,\|u\|_{\ell^p},\quad
\|(-\Delta)^s u\|_{\ell^p}\le\frac{2^{1+s}d^s}{(1-s)\Gamma(1-s)}\|u\|_{\ell^p}.
\]
\end{proposition}

\begin{proof}
Using
\[
\Delta u(x)=\sum_{j=1}^d\bigl(u(x+e_j)+u(x-e_j)-2u(x)\bigr),
\]
we get by triangle inequality and translation invariance of the \(\ell^p\)-norm:
\[
\|\Delta u\|_{\ell^p}
\le \sum_{j=1}^d\Bigl(\|u(\cdot+e_j)\|_{\ell^p}
+\|u(\cdot-e_j)\|_{\ell^p}+2\|u\|_{\ell^p}\Bigr)
=4d\,\|u\|_{\ell^p}.
\]

By the Bochner formula,
\[
(-\Delta)^s u=\frac{s}{\Gamma(1-s)}
\int_0^\infty \bigl(u-e^{t\Delta}u\bigr)\,t^{-1-s}\,dt.
\]
For the heat semigroup \((e^{t\Delta})_{t\ge0}\), we use that it is an
\(\ell^p\)-contraction:
\[
\|e^{t\Delta}u\|_{\ell^p}\le \|u\|_{\ell^p},\qquad t\ge0.
\]
Hence
\[
\|u-e^{t\Delta}u\|_{\ell^p}\le 2\|u\|_{\ell^p}.
\]
Also,
\[
u-e^{t\Delta}u
=-\int_0^t \frac{d}{d\tau}\bigl(e^{\tau\Delta}u\bigr)\,d\tau
=-\int_0^t \Delta e^{\tau\Delta}u\,d\tau,
\]
thus
\[
\|u-e^{t\Delta}u\|_{\ell^p}
\le \int_0^t \|\Delta e^{\tau\Delta}u\|_{\ell^p}\,d\tau
\le 4d\,t\,\|u\|_{\ell^p}.
\]
Therefore
\[
\|u-e^{t\Delta}u\|_{\ell^p}
\le \min\{2,\,4d\,t\}\,\|u\|_{\ell^p}.
\]
So
\[
\|(-\Delta)^s u\|_{\ell^p}
\le \frac{s}{\Gamma(1-s)}
\left(\int_0^\infty \min\{2,4d\,t\}\,t^{-1-s}\,dt\right)\|u\|_{\ell^p}= \frac{2^{1+s}d^s}{(1-s)\Gamma(1-s)}\|u\|_{\ell^p},
\]
which proves boundedness on \(\ell^p(\mathbb Z^d)\), \(1\le p\le\infty\).
\end{proof}

Next, we introduce the Fourier transform on the lattice $\mathbb Z^d$. For any $u\in \ell^{1}(\mathbb{Z}^d)$, define its Fourier transform
$\widehat u:\mathbb{T}^d\to\mathbb{C}$ by (see \cite[Chapter 3]{Grafakos2014})
\[
\widehat u(\xi)
=\frac{1}{(2\pi)^{d/2}}
\sum_{x\in\mathbb{Z}^d}u(x)e^{-i\langle x,\xi\rangle},
\qquad
\xi\in[-\pi,\pi]^d,
\]
where
\[
\langle x,\xi\rangle=\sum_{j=1}^d x_j\,\xi_j.
\]
Hence $\widehat u\in L^\infty(\mathbb{T}^d)$ and
\[
\|\widehat u\|_{L^\infty(\mathbb{T}^d)}
\le (2\pi)^{-d/2}\|u\|_{\ell^1(\mathbb{Z}^d)}.
\]

The inversion formula is
\[
u(x)
=\frac{1}{(2\pi)^{d/2}}
\int_{\mathbb{T}^d}\widehat u(\xi)\,e^{i\langle x,\xi\rangle}\,d\xi,
\qquad x\in\mathbb{Z}^d.
\]
Moreover, for $u,v\in \ell^1(\mathbb{Z}^d)$ one has
\[
\widehat{u*v}(\xi)=(2\pi)^{d/2}\,\widehat u(\xi)\widehat v(\xi),
\]
where the discrete convolution is
\[
(u*v)(x):=\sum_{y\in\mathbb{Z}^d}u(x-y)v(y).
\]

The transform extends by density to a unitary map
\[
\mathcal{F}:\ell^2(\mathbb{Z}^d)\longrightarrow L^2(\mathbb{T}^d),
\]
with Plancherel identity
\[
\|u\|_{\ell^2(\mathbb{Z}^d)}=\|\widehat u\|_{L^2(\mathbb{T}^d)},
\]
and Parseval identity
\[
\sum_{x\in\mathbb{Z}^d}u(x)\overline{v(x)}
=
\int_{\mathbb{T}^d}\widehat u(\xi)\,\overline{\widehat v(\xi)}\,d\xi.
\]

For later use, we also record the translation and difference multipliers.
If $\tau_j u(x):=u(x+e_j)$, then
\[
\widehat{\tau_j u}(\xi)=e^{i\xi_j}\widehat u(\xi).
\]
Hence, for the discrete Laplacian, its Fourier symbol is
\[
\widehat{-\Delta u}(\xi)=
\omega(\xi)\widehat u(\xi),
\qquad
\omega(\xi):=4\sum_{i=1}^d\sin^2\!\left(\frac{\xi_i}{2}\right).
\]
Similarly, as shown in \cite{chen2025logarithmic}, for every $u\in \ell^1(\mathbb{Z^d)},$
\[
\widehat{(-\Delta )^s u}(\xi)=
\omega(\xi)^s\widehat u(\xi).
\]
Consequently, spectral multipliers of $-\Delta$ and $(-\Delta)^s$ are diagonalized on $\mathbb{T}^d$ via multiplication with functions of $\omega(\xi)$ and $\omega(\xi)^s$, respectively.

We now present a unified treatment of the semigroup generated by $(-\Delta)^s$ on $\mathbb Z^d$, for all $s\in(0,1]$. By Proposition~\ref{prop:lp-boundedness-lap-frac}, we have
\[
(-\Delta)^s\in \mathcal L(\ell^p(\mathbb Z^d)),\qquad 1\le p\le\infty.
\]
Hence, by the bounded-generator semigroup construction (cf. \cite[Chapter 2.1]{keller2021graphs}), for every \(t\ge0\) we define
\begin{equation}\label{semig}
    S_s(t):=e^{-t(-\Delta)^s}
=\sum_{n=0}^\infty \frac{(-t)^n}{n!}\bigl((-\Delta)^s\bigr)^n,
\end{equation}
and \(\{S_s(t)\}_{t\ge0}\) is a strongly continuous contraction semigroup on \(\ell^p(\mathbb Z^d),1\le p\le \infty\), with
\[
\partial_t S_s(t)u
=
-(-\Delta)^s S_s(t)u
=
-S_s(t)(-\Delta)^s u,
\qquad t\ge0.
\]

Therefore, for any initial datum \(u_0\in\ell^p(\mathbb Z^d)\), the function
\[
u(t):=S_s(t)u_0
\]
solves
\[
\partial_t u(t)+(-\Delta)^s u(t)=0,\qquad u(0)=u_0,
\]
in the strong sense on \(\ell^p(\mathbb Z^d)\).  
Moreover, uniqueness for this Cauchy problem follows from the uniqueness-of-semigroups principle for bounded generators (cf. \cite[Exercise 2.8]{keller2021graphs}): any solution with the same initial datum must coincide with \(S_s(t)u_0\).

Using
\[
\widehat{(-\Delta)^s u}(\xi)=\omega(\xi)^s\widehat u(\xi),
\qquad
\omega(\xi):=4\sum_{j=1}^d \sin^2\!\left(\frac{\xi_j}{2}\right),
\]
we obtain
\[
\widehat{S_s(t)u}(\xi)=e^{-t\omega(\xi)^s}\widehat u(\xi).
\]
Hence, \(S_s(t)\) is a convolution operator:
\[
S_s(t)u=G_t^{(s)}*u,
\]
where
\[    G_t^{(s)}(x)
=
\frac{1}{(2\pi)^{d}}
\int_{\mathbb T^d}
e^{-t\omega(\xi)^s}e^{i\langle x,\xi\rangle}\,d\xi,
\qquad x\in\mathbb Z^d.\]

Moreover, 
\begin{equation}\label{mass}
    \sum_{x\in\mathbb Z^d}G_t^{(s)}(x)
=
(2\pi)^{d/2}\widehat{G_t^{(s)}}(0)
=1,\qquad t>0.
\end{equation}
 Consequently, \(S_s(t)\) is mass preserving: for every \(u_0\in\ell^1(\mathbb Z^d)\),
\[
\sum_{x\in\mathbb Z^d}(S_s(t)u_0)(x)=\sum_{x\in\mathbb Z^d}u_0(x).
\]

    
\subsection{Large-time Asymptotic Convergence}

In this subsection, we establish the long-time convergence of lattice fractional-diffusion solutions toward the fundamental kernel profile, together with quantitative rates under finite first-moment assumptions and qualitative convergence in the general \(\ell^1\) setting.

Let \(d\ge 1\), \(s\in(0,1]\). For $u_0\in \ell^1(\mathbb{Z}^d)$, consider the following Cauchy problem on \(\mathbb Z^d\):
\begin{equation}
\begin{cases}
\partial_t u(t,x)+(-\Delta)^s u(t,x)=0, & t>0,\ x\in\mathbb Z^d,\\[2mm]
u(0,x)=u_0(x), & x\in\mathbb Z^d.
\end{cases}
\end{equation}
The above problem admits a unique solution in \(\ell^1(\mathbb Z^d)\), given by
\[
u(t,x)=S_s(t)u_0=e^{-t(-\Delta)^s}u_0,\qquad t\ge0.
\]

Recall that
\[
G_t^{(s)}(x):=\frac{1}{(2\pi)^{d}}\int_{\mathbb T^d}e^{-t\omega(\xi)^s}e^{i\langle x,\xi\rangle}\,d\xi,
\qquad
\omega(\xi)=4\sum_{j=1}^d\sin^2\!\Bigl(\frac{\xi_j}{2}\Bigr).
\]
Then the solution
\[
u(t,x)=G_t^{(s)}*u_0=\sum_{y\in\mathbb Z^d}G_t^{(s)}(x-y)\,u_0(y).
\]
For \(u_0\in \ell^1(\mathbb Z^d)\), define
\[
M:=\sum_{x\in\mathbb Z^d}u_0(x),
\qquad
\mathcal N_1(u_0):=\sum_{x\in\mathbb Z^d}|x|\,|u_0(x)|.
\]

\subsubsection{Pointwise and \texorpdfstring{$\ell^1$}{l1} Bounds for Kernel Differences}
In this part, we derive sharp pointwise and \(\ell^1\)-difference estimates for translated lattice heat kernels, which provide the core quantitative input for the subsequent asymptotic analysis.

\begin{proof}[\bf Proof of Proposition \ref{prop:kernel-increment-pointwise}.]
Fix \(j\in\{1,\dots,d\}\). By Fourier inversion (\ref{gaussian}),
\[
G_t^{(s)}(x+e_j)-G_t^{(s)}(x)
=
\frac1{(2\pi)^{d}}\int_{\mathbb T^d}
\bigl(e^{i\xi_j}-1\bigr)e^{-t\omega(\xi)^s}e^{i\langle x,\xi\rangle}\,d\xi.
\]
Hence
\[\sup_{x\in\mathbb Z^d}|G_t^{(s)}(x+e_j)-G_t^{(s)}(x)|
\le
\frac1{(2\pi)^{d}}\int_{\mathbb T^d}|e^{i\xi_j}-1|\,e^{-t\omega(\xi)^s}\,d\xi.\]
Using \(|e^{i\theta}-1|\le |\theta|\), we have
\[
|e^{i\xi_j}-1|\le |\xi|,\qquad \xi\in\mathbb T^d.
\]
By the fact that
\[
\omega(\xi)\ge \frac{4}{\pi^2}|\xi|^2,\qquad \xi\in\mathbb T^d,
\]
we obtain
\[
\sup_{x\in\mathbb Z^d}|G_t^{(s)}(x+e_j)-G_t^{(s)}(x)|
\le
\frac1{(2\pi)^{d}}\int_{\mathbb T^d} |\xi|\,e^{-\frac{4^s}{\pi^{2s}} t|\xi|^{2s}}\,d\xi.
\]
    Now set \(\eta=t^{1/(2s)}\xi\). Then:
\[\int_{\mathbb T^d} |\xi|\,e^{-\frac{4^s}{\pi^{2s}} t|\xi|^{2s}}\,d\xi
=
t^{-\frac{d+1}{2s}}
\int_{t^{1/(2s)}\mathbb T^d}
|\eta|\,e^{-\frac{4^s}{\pi^{2s}} |\eta|^{2s}}\,d\eta
\le
C\,t^{-\frac{d+1}{2s}}.
\]
Hence
\begin{equation}
\label{eq:unit-step-bound-final}
\sup_{x\in\mathbb Z^d}|G_t^{(s)}(x+e_j)-G_t^{(s)}(x)|
\le C\,t^{-\frac{d+1}{2s}}.
\end{equation}

Given \(y\in\mathbb Z^d\), choose a nearest-neighbor path
\[
x=x^{(0)},x^{(1)},\dots,x^{(m)}=x-y,
\]
with \(m=\sum_{i=1}^d|y_i|\) and \(x^{(k+1)}-x^{(k)}\in\{\pm e_1,\dots,\pm e_d\}\).
Then
\[
G_t^{(s)}(x-y)-G_t^{(s)}(x)
=
\sum_{k=0}^{m-1}\Bigl(G_t^{(s)}(x^{(k+1)})-G_t^{(s)}(x^{(k)})\Bigr),
\]
so by \eqref{eq:unit-step-bound-final},
\[
|G_t^{(s)}(x-y)-G_t^{(s)}(x)|
\le
m\,C\,t^{-\frac{d+1}{2s}}.
\]
Finally, \(m\le \sqrt d\,|y|\), hence
\[
|G_t^{(s)}(x-y)-G_t^{(s)}(x)|
\le
C\,|y|\,t^{-\frac{d+1}{2s}}.
\]
From Fourier inversion,
\[
|G_t^{(s)}(x)|
\le
\frac1{(2\pi)^{d/2}}\int_{\mathbb T^d}e^{-t\omega(\xi)^s}\,d\xi.
\]
Using \(\omega(\xi)\ge  \frac{4}{\pi^2}|\xi|^2\) on \(\mathbb T^d\), we get
\[
\sup_{x\in\mathbb Z^d}|G_t^{(s)}(x)|
\le C\,t^{-d/(2s)}.
\]
Hence
\[|G_t^{(s)}(x-y)-G_t^{(s)}(x)|
\le 2\sup_z|G_t^{(s)}(z)|
\le C\,t^{-d/(2s)}.\]
From the two estimates,
\[
|G_t^{(s)}(x-y)-G_t^{(s)}(x)|
\le
C\,\min\!\left\{t^{-d/(2s)},\,|y|\,t^{-(d+1)/(2s)}\right\}
=
C\,t^{-d/(2s)}\min\!\left\{1,\ |y|\,t^{-1/(2s)}\right\}.
\]
This is \eqref{eq:linfty-difference-kernel-sharp}.
\end{proof}

The following Sobolev estimate on the torus Fourier side will be used repeatedly to convert \(W^{n,1}\)-regularity of multipliers into summability of their lattice kernels.

\begin{lemma}
\label{lem:Wm1-to-l1}
Let \(n\in\mathbb N\) with \(n>d\), and let \(m\in W^{n,1}(\mathbb T^d)\).
Define 
\[
a(x):=\int_{\mathbb T^d} m(\xi)e^{i x\cdot \xi}\,d\xi,
\qquad x\in\mathbb Z^d.
\]
Then
\[
\|a\|_{\ell^1(\mathbb Z^d)}
\le C_{d,n}\sum_{|\alpha|\le n}\|\partial^\alpha m\|_{L^1(\mathbb T^d)}.
\]
\end{lemma}

\begin{proof}
Set
\[
\Lambda_n(x):=\sum_{|\alpha|\le n}|x^\alpha|, \qquad x\in\mathbb Z^d.
\]
Since \(n>d\), there exists \(C=C(d,n)\) such that
\begin{equation}
\label{eq:weight-summable}
\sum_{x\in\mathbb Z^d}\frac1{1+\Lambda_n(x)}\le C<\infty .
\end{equation}

For each multi-index \(\alpha\), integration by parts on the torus gives
\[
x^\alpha a(x)
=
i^{|\alpha|}\int_{\mathbb T^d}\partial^\alpha m(\xi)\,e^{i x\cdot \xi}\,d\xi.
\]
Hence
\[
|x^\alpha|\,|a(x)|
\le \|\partial^\alpha m\|_{L^1(\mathbb T^d)}.
\]
Summing over \(|\alpha|\le n\),
\[
\Lambda_n(x)\,|a(x)|
\le \sum_{|\alpha|\le n}\|\partial^\alpha m\|_{L^1(\mathbb T^d)}.
\]
Therefore
\[
(1+\Lambda_n(x))\,|a(x)|
\le 2\sum_{|\alpha|\le n}\|\partial^\alpha m\|_{L^1(\mathbb T^d)},
\]
so
\[
|a(x)|
\le \frac{2}{1+\Lambda_n(x)}
\sum_{|\alpha|\le n}\|\partial^\alpha m\|_{L^1(\mathbb T^d)}.
\]
Now sum in \(x\in\mathbb Z^d\), use \eqref{eq:weight-summable}, and conclude
\[
\|a\|_{\ell^1(\mathbb Z^d)}
\le
C\sum_{|\alpha|\le n}\|\partial^\alpha m\|_{L^1(\mathbb T^d)}.
\]
This proves the lemma.
\end{proof}

The next proposition is the key translation-increment estimate for the lattice heat kernel; it provides the sharp \(\ell^1\)-control that drives the first-order large-time asymptotics.

\begin{proof}[\bf Proof of Proposition \ref{prop:l1-sharp-by-splitting}.]
Fix \(j\in\{1,\dots,d\}\), and define
\[
D_{j,t}(x):=G_t^{(s)}(x+e_j)-G_t^{(s)}(x).
\]
Then
\[
D_{j,t}(x)=\frac1{(2\pi)^{d}}\int_{\mathbb T^d}m_{j,t}(\xi)e^{i\langle x,\xi\rangle}\,d\xi,
\qquad m_{j,t}(\xi)=(e^{i\xi_j}-1)e^{-t\omega(\xi)^s}.
\]
We first prove
\begin{equation}
\label{eq:unit-step-sharp-rig}
\|D_{j,t}\|_{\ell^1(\mathbb Z^d)}\le C\,t^{-1/(2s)},\qquad t\ge1.
\end{equation}

Choose \(\rho\in C_c^\infty(\mathbb R^d)\), \(0\le \rho\le1\), with
\[
\rho(\xi)=1 \ \text{for } |\xi|\le1,\qquad
\rho(\xi)=0 \ \text{for } |\xi|\ge2,
\]
and define
\[
\rho_t(\xi):=\rho\,\big(t^{1/(2s)}\xi\big),\qquad
m_{j,t}^{\mathrm{low}}:=m_{j,t}\rho_t,\qquad
m_{j,t}^{\mathrm{high}}:=m_{j,t}(1-\rho_t).
\]
Let \(D_{j,t}^{\mathrm{low}}\) and  \(D_{j,t}^{\mathrm{high}}\) be the inverse Fourier transforms of
\(m_{j,t}^{\mathrm{low}}\) and \(m_{j,t}^{\mathrm{high}}\), respectively. Then
\[
D_{j,t}=D_{j,t}^{\mathrm{low}}+D_{j,t}^{\mathrm{high}},
\qquad
\|D_{j,t}\|_{\ell^1}\le
\|D_{j,t}^{\mathrm{low}}\|_{\ell^1}+\|D_{j,t}^{\mathrm{high}}\|_{\ell^1}.
\]

\medskip
\noindent
\textbf{(i)  low-frequency part.}
Since
\[
\operatorname{supp} m_{j,t}^{\mathrm{low}}\subset\{|\xi|\le2t^{-1/(2s)}\},
\]
and near \(0\), \(|e^{i\xi_j}-1|\lesssim|\xi_j|\), \(\omega(\xi)\sim |\xi|^2\), we can write
\begin{equation}\label{mjtlow}
    m_{j,t}^{\mathrm{low}}(\xi)
=
t^{-1/(2s)}\widetilde m_t\,\big(t^{1/(2s)}\xi\big),
\end{equation}
where
\[
\widetilde m_t(\eta):=
\Big(\frac{e^{i t^{-1/(2s)}\eta_j}-1}{t^{-1/(2s)}}\Big)
e^{-t\omega(t^{-1/(2s)}\eta)^s}\rho(\eta),
\qquad
\operatorname{supp}\widetilde m_t\subset\{|\eta|\le2\}.
\]
We claim that for each fixed integer $N$ and multi-index \(|\beta|\le N<d+2s+1\), we have
\begin{equation}\label{chja}
\sup_{t\ge1}\|\partial^\beta \widetilde m_t\|_{L^1(\mathbb R^d)}<\infty .
\end{equation}
Write
\[
\widetilde m_t(\eta)=A_t(\eta)\,B_t(\eta)\,\rho(\eta),\qquad
A_t(\eta):=\frac{e^{i\lambda\eta_j}-1}{\lambda},\quad
B_t(\eta):=e^{-t\omega(\lambda\eta)^s},\quad
\lambda=t^{-1/(2s)}\in(0,1].
\]
By Leibniz formula,
\[
\partial^\beta \widetilde m_t
=
\sum_{\beta_1+\beta_2+\beta_3=\beta}
C_{\beta_1,\beta_2,\beta_3}\,
(\partial^{\beta_1}A_t)\,(\partial^{\beta_2}B_t)\,(\partial^{\beta_3}\rho).
\]
Hence
\[
\|\partial^\beta \widetilde m_t\|_{L^1(\mathbb R^d)}
\le
C_\beta\sum_{\beta_1+\beta_2+\beta_3=\beta}
\|\partial^{\beta_1}A_t\|_{L^\infty(B_2)}
\,
\|\partial^{\beta_2}B_t\|_{L^1(B_2)}.
\]
For any multi-index \(\beta=(\beta_1,\dots,\beta_d)\), we have \[ \partial_\eta^\beta A_t(\eta)=0 \quad\text{whenever}\quad \beta_k\neq0\ \text{for some }k\neq j. \] Therefore, it is enough to consider multi-indices of the form $\beta=\beta_j e_j.$ In that case, \[ \partial_{\eta_j}^{\beta_j}A_t(\eta)= \begin{cases} \dfrac{e^{i\lambda\eta_j}-1}{\lambda}, & \beta_j=0,\\[0.8em] i^{\beta_j}\lambda^{\beta_j-1}e^{i\lambda\eta_j}, & \beta_j\ge1. \end{cases} \] Hence \[ \sup_{t\ge1}\sup_{|\eta|\le2}|\partial^\beta A_t(\eta)|\le C_\beta. \]

Set \(\Phi_t(\eta):=t\omega(\lambda\eta)^s\), so \(B_t=e^{-\Phi_t}\).
By Fa\`a di Bruno formula, for each multi-index \(\kappa\),
\[
\partial^\kappa B_t
=
e^{-\Phi_t}
\sum_{q=1}^{|\kappa|}
\ \sum_{\substack{\gamma_1+\cdots+\gamma_q=\kappa\\ |\gamma_\ell|\ge1}}
C_{\kappa,\gamma_1,\dots,\gamma_q}
\prod_{\ell=1}^q \partial^{\gamma_\ell}\Phi_t,
\qquad \kappa\neq0.
\]
Now
\[
\partial^\nu\Phi_t(\eta)
=
t\,\lambda^{|\nu|}\,\partial^\nu(\omega^s)(\lambda\eta).
\]
Using \(\omega(\xi)\sim|\xi|^2\) near \(0\), one has on \(|\xi|\le2\):
\[
|\partial^\nu(\omega^s)(\xi)|\le C_\nu |\xi|^{2s-|\nu|}.
\]
Therefore, for \(|\eta|\le2\)
\[
|\partial^\nu\Phi_t(\eta)|
\le C_\nu\, t\,\lambda^{|\nu|}(\lambda|\eta|)^{2s-|\nu|}
= C_\nu\,|\eta|^{2s-|\nu|},
\]
since \(t\lambda^{2s}=1\). Hence each term in \(\partial^\kappa B_t\) is bounded by
\[
\prod_{\ell=1}^q |\eta|^{2s-|\gamma_\ell|}
=|\eta|^{2sq-|\kappa|}.
\]
Now return to
\[
\partial^\beta \widetilde m_t
=
\sum_{\beta_1+\beta_2+\beta_3=\beta}
C_{\beta_1,\beta_2,\beta_3}\,
(\partial^{\beta_1}A_t)(\partial^{\beta_2}B_t)(\partial^{\beta_3}\rho).
\]
The worst singular term is \((\beta_1,\beta_2,\beta_3)=(0,\beta,0)\):
\[
|A_t(\eta)\,\partial^\beta B_t(\eta)|
\lesssim |\eta|\cdot |\eta|^{2s-|\beta|}
=|\eta|^{2s+1-|\beta|}.
\]
Hence this term is \(L^1(B_2)\) provided
\[
2s+1-|\beta|>-d
\quad\Longleftrightarrow\quad
|\beta|<d+2s+1.
\]
All other Leibniz terms are less singular. Consequently, for every fixed integer
\(N<d+2s+1\),
\[
\sup_{t\ge1}\sum_{|\beta|\le N}\|\partial^\beta \widetilde m_t\|_{L^1(\mathbb R^d)}<\infty,
\]
which is exactly \eqref{chja}.

Let
\begin{equation}\label{kty}
    K_t(y):
=\frac1{(2\pi)^{d}}\int_{\mathbb R^d}\widetilde m_t(\eta)e^{iy\cdot\eta}\,d\eta.
\end{equation}
Hence, by integration by parts, for \(|\beta|=N\),
\[
y^\beta K_t(y)
= c_\beta\int_{\mathbb R^d}\partial^\beta\widetilde m_t(\eta)e^{iy\cdot\eta}\,d\eta.
\]
Summing over \(|\beta|=N\), we get
\begin{equation}\label{esy1}
    |K_t(y)|\le C_N(1+|y|)^{-N},
\end{equation}
with \(C_N\) independent of \(t\ge1\).

Set \(\eta=t^{1/2s}\xi\), by (\ref{mjtlow}), (\ref{kty}) and $\operatorname{supp}\widetilde m_t\subset\{|\eta|\le2\},$
\[
D_{j,t}^{\mathrm{low}}(x)
= \frac1{(2\pi)^{d}}\,t^{-1/(2s)}t^{-d/2s}
\int_{\mathbb R^d}\widetilde m_t(\eta)e^{i(x\cdot\eta)t^{-1/2s}}\,d\eta
= t^{-(d+1)/(2s)}K_t\,\Big(t^{-1/2s}x\Big).
\]
Therefore, by \eqref{esy1}
\[
|D_{j,t}^{\mathrm{low}}(x)|
\le C_N\,t^{-(d+1)/(2s)}
\Big(1+t^{-1/(2s)}|x|\Big)^{-N},\quad t\ge 1
\]
Choose \(N=d+1\). Using the fact
\[
\sum_{x\in\mathbb Z^d}(1+a|x|)^{-N}\le C_{d,N}a^{-d},\qquad 0<a\le1,
\]
with \(a=t^{-1/(2s)}\), we obtain
\[
\|D_{j,t}^{\mathrm{low}}\|_{\ell^1}
\le
C\,t^{-(d+1)/(2s)}
\sum_{x\in\mathbb Z^d}\Big(1+t^{-1/(2s)}|x|\Big)^{-N}
\le C\,t^{-1/(2s)}.
\]

\medskip
\noindent
\textbf{(ii)  high-frequency part.}

To avoid derivative loss coming from the \(t\)-dependent cutoff \(\rho_t\), we decompose
\[
m_{j,t}^{\mathrm{high}}
=(e^{i\xi_j}-1)e^{-t\omega(\xi)^s}(1-\rho_t(\xi))
\]
as
\[
m_{j,t}^{\mathrm{high}}
=
(e^{i\xi_j}-1)e^{-t\omega(\xi)^s}(1-\chi(\xi))
+
(e^{i\xi_j}-1)e^{-t\omega(\xi)^s}\big(\chi(\xi)-\rho_t(\xi)\big):=m_{j,t}^{(1)}+m_{j,t}^{(2)},
\]
where
\[
\chi(\xi):=\rho(\xi/\delta),\qquad\delta>1\quad\text{fixed}.
\]
Then
\[
\chi(\xi)=1 \ \text{for } |\xi|\le \delta,\qquad
\chi(\xi)=0 \ \text{for } |\xi|\ge 2\delta.
\]

On \(\operatorname{supp}(1-\chi)\), \(|\xi|\ge\delta\), hence
\[
\omega(\xi)^s\ge \mu_\delta>0.
\]
By Leibniz formula and Fa\`a di Bruno formula, for \(|\alpha|\le M\), it's easy to obtain
\[
\|\partial^\alpha m_{j,t}^{(1)}\|_{L^1(\mathbb T^d)}
\le C_{\alpha,M,\delta}\,e^{-c_\delta t}.
\]
Therefore, by Lemma \ref{lem:Wm1-to-l1}
\[
\|\mathcal F^{-1}m_{j,t}^{(1)}\|_{\ell^1(\mathbb Z^d)}\le C e^{-c_\delta t}
\le C t^{-1/(2s)}.
\]

Next, we consider the estimate of \(m_{j,t}^{(2)}\).
Its support is contained in
\[
\{\,t^{-1/(2s)}\le |\xi|\le 2\delta\,\}.
\]
Note that \(m_{j,t}^{(2)}\) has the same kernel scaling as the low-frequency piece, thus, for every \(N\),
\[
\big|\mathcal F^{-1}m_{j,t}^{(2)}(x)\big|
\le C_N\,t^{-(d+1)/(2s)}\Big(1+t^{-1/(2s)}|x|\Big)^{-N},
\]
we obtain
\[
\|\mathcal F^{-1}m_{j,t}^{(2)}\|_{\ell^1}\le C t^{-1/(2s)}.
\]
Combining both parts,
\[
\|D_{j,t}^{\mathrm{high}}\|_{\ell^1}
\le C
\|\mathcal F^{-1}m_{j,t}^{(1)}\|_{\ell^1}
+C\|\mathcal F^{-1}m_{j,t}^{(2)}\|_{\ell^1}
\le C t^{-1/(2s)}.
\]
Therefore,
\[
\|D_{j,t}\|_{\ell^1}
\le
\|D_{j,t}^{\mathrm{low}}\|_{\ell^1}
+\|D_{j,t}^{\mathrm{high}}\|_{\ell^1}
\le C t^{-1/(2s)},\qquad t\ge1,
\]
which proves \eqref{eq:unit-step-sharp-rig}.

\medskip
Now let \(y\in\mathbb Z^d\). Choose a nearest-neighbor path
\[
0=z^{(0)},z^{(1)},\dots,z^{(m)}=y,\qquad
m=\sum_{i=1}^d|y_i|,\quad z^{(k+1)}-z^{(k)}\in\{\pm e_1,\dots,\pm e_d\}.
\]
By telescoping,
\[
G_t^{(s)}(x-y)-G_t^{(s)}(x)
=
\sum_{k=0}^{m-1}\Big(G_t^{(s)}(x-z^{(k+1)})-G_t^{(s)}(x-z^{(k)})\Big).
\]
Taking \(\ell^1\)-norm in \(x\), and using translation invariance,
\[
\sum_x |G_t^{(s)}(x-y)-G_t^{(s)}(x)|
\le
\sum_{k=0}^{m-1}\|G_t^{(s)}(\cdot-\delta_k)-G_t^{(s)}\|_{\ell^1}
\le
C\,m\,t^{-1/(2s)}
\le
C\,|y|\,t^{-1/(2s)},\quad t\ge 1.
\]
Also, since \(G_t^{(s)}\ge0\) and \(\sum_x G_t^{(s)}(x)=1\),
\[
\sum_x |G_t^{(s)}(x-y)-G_t^{(s)}(x)|
\le
\sum_x G_t^{(s)}(x-y)+\sum_x G_t^{(s)}(x)=2.
\]
Hence
\[
\sum_x |G_t^{(s)}(x-y)-G_t^{(s)}(x)|
\le
C\min\{1,\ |y|t^{-1/(2s)}\},
\]
which is \eqref{eq:l1-sharp-main}.
\end{proof}

\subsubsection{Sharpness of the Convergence with Finite First Moment}
We now show that, within the class of initial data with finite first moment, the decay rate obtained above is optimal and cannot be improved in general.

	\begin{proof}[\bf Proof of Theorem \ref{thm:lattice-first-order}.]
Since
\[
u(t,x)=\sum_{z\in\mathbb Z^d}G_t^{(s)}(x-z)\,u_0(z),\qquad
M=\sum_{z\in\mathbb Z^d}u_0(z),
\]
then
\[
u(t,x)-M G_t^{(s)}(x)
=
\sum_{z\in\mathbb Z^d}u_0(z)\Big(G_t^{(s)}(x-z)-G_t^{(s)}(x)\Big).
\]
Hence, by triangle inequality,
\begin{equation}\label{eq:remainder-repr}
|u(t,x)-M G_t^{(s)}(x)|
\le
\sum_{z\in\mathbb Z^d}|u_0(z)|\,\big|G_t^{(s)}(x-z)-G_t^{(s)}(x)\big|.
\end{equation}

\noindent
\textbf{(i) \(\ell^\infty\)-estimate.}
Taking supremum in \(x\) in \eqref{eq:remainder-repr} and using Proposition~\ref{prop:kernel-increment-pointwise},
\[
\|u(t)-M G_t^{(s)}\|_{\ell^\infty}
\le
\sum_{z}|u_0(z)|\,
\sup_x |G_t^{(s)}(x-z)-G_t^{(s)}(x)|
\]
\[
\le
C\,t^{-d/(2s)}\sum_{z}|u_0(z)|
\min\!\left\{1,\ |z|\,t^{-1/(2s)}\right\}\le C\,\mathcal N_1(u_0)\,t^{-(d+1)/(2s)}.
\]
Multiplying by \(t^{d/(2s)}\), we obtain
\[
t^{d/(2s)}\|u(t)-M G_t^{(s)}\|_{\ell^\infty}
\le
C\,\mathcal N_1(u_0)\,t^{-1/(2s)},
\]
which is \eqref{eq:lattice-Linfty-rate}.

\medskip
\noindent
\textbf{(ii) \(\ell^1\)-estimate.}
Sum \eqref{eq:remainder-repr} over \(x\in\mathbb Z^d\), then apply Fubini theorem:
\[
\|u(t)-M G_t^{(s)}\|_{\ell^1}
\le
\sum_{z}|u_0(z)|
\sum_{x}\big|G_t^{(s)}(x-z)-G_t^{(s)}(x)\big|.
\]
Using Proposition~\ref{prop:l1-sharp-by-splitting},
\[
\|u(t)-M G_t^{(s)}\|_{\ell^1}
\le
C\sum_{z}|u_0(z)|\min\!\left\{1,\ |z|\,t^{-1/(2s)}\right\}
\le
C\,t^{-1/(2s)}\sum_{z}|z|\,|u_0(z)|.
\]
Therefore
\[
\|u(t)-M G_t^{(s)}\|_{\ell^1}
\le
C\,\mathcal N_1(u_0)\,t^{-1/(2s)},
\]
which is \eqref{eq:lattice-L1-rate}.

\medskip
\noindent
\textbf{(iii) \(\ell^p\)-estimate for \(1\le p\le\infty\).}
Let
\[
R(t):=u(t)-M G_t^{(s)}.
\]
By discrete interpolation theorem,
\[
\|R(t)\|_{\ell^p}
\le
\|R(t)\|_{\ell^1}^{1/p}\,
\|R(t)\|_{\ell^\infty}^{1-1/p}.
\]
Insert the bounds from Steps 1--2:
\[
\|R(t)\|_{\ell^p}
\le
C\,\mathcal N_1(u_0)\,
t^{-\frac1{2s}\frac1p}\,
t^{-\frac{d+1}{2s}\left(1-\frac1p\right)}
=
C\,\mathcal N_1(u_0)\,
t^{-\frac{d}{2s}\left(1-\frac1p\right)}\,t^{-1/(2s)}.
\]
Hence
\[
t^{\frac d{2s}\left(1-\frac1p\right)}
\|u(t)-M G_t^{(s)}\|_{\ell^p}
\le
C\,\mathcal N_1(u_0)\,t^{-1/(2s)},
\]
which is \eqref{eq:lattice-Lp-rate}. This completes the proof.
\end{proof}

Next, we show that the lattice kernel  satisfies a limit scaling toward the stable profile in self-similar variables.

\begin{proof}[\bf Proof of Proposition \ref{lem:lattice-stable-scaling}.]
Fix \(R>0\), and set $\lambda=t^{-1/(2s)}.$ By the change of variables \(\xi=\lambda\zeta\),
\[
t^{d/(2s)}G_t^{(s)}\!\big(x_t(\eta)\big)
=
\frac1{(2\pi)^d}
\int_{[-\pi/\lambda,\pi/\lambda]^d}
e^{-t\omega(\lambda\zeta)^s}
e^{i\lambda x_t(\eta)\cdot\zeta}\,d\zeta
=:I_t(\eta).
\]
We prove \(\sup_{|\eta|\le R}|I_t(\eta)-\Phi_s(\eta)|\to0\) as $t\rightarrow \infty.$

\medskip
\noindent
\textbf{(i) pointwise convergence of the integrand.}
Fix \(\zeta\in\mathbb R^d\). Since \(\omega(\xi)=|\xi|^2+O(|\xi|^4)\) as \(\xi\to0\),
\[
t\,\omega(\lambda\zeta)^s
=
\lambda^{-2s}\,\omega(\lambda\zeta)^s
\longrightarrow |\zeta|^{2s}\quad \text{as}\quad t\rightarrow \infty.
\]
From \eqref{eq:nearest-lattice-selection},
\[
\lambda x_t(\eta)-\eta
=
\lambda\big(x_t(\eta)-t^{1/(2s)}\eta\big),
\]
hence
\[
\sup_{|\eta|\le R}|\lambda x_t(\eta)-\eta|
\le C_0\lambda\to0\quad \text{as}\quad t\rightarrow \infty.
\]
Therefore, for each fixed \(\zeta\),
\[
e^{-t\omega(\lambda\zeta)^s}e^{i\lambda x_t(\eta)\cdot\zeta}
\to
e^{-|\zeta|^{2s}}e^{i\eta\cdot\zeta},
\]
uniformly in \(|\eta|\le R\). 

\medskip
\noindent
\textbf{(ii) uniform tail control.}
Because \(\omega(\theta)\ge \frac{4}{\pi^2}|\theta|^2\) on \([-\pi,\pi]^d\), we have
\[
t\omega(\lambda\zeta)^s\ge (\frac{4}{\pi^2})^s|\zeta|^{2s}
\quad\text{whenever } |\lambda\zeta_j|\le \pi \text{ for all }j,
\]
thus
\[
\left|e^{-t\omega(\lambda\zeta)^s}e^{i\lambda x_t(\eta)\cdot\zeta}\right|
\le e^{-\frac{4}{\pi^2}|\zeta|^{2s}}.
\]
Hence, for any \(A>0\),
\[
\sup_{t\ge1}\sup_{|\eta|\le R}
\int_{\{|\zeta|>A\}\cap[-\pi/\lambda,\pi/\lambda]^d}
\left|e^{-t\omega(\lambda\zeta)^s}e^{i\lambda x_t(\eta)\cdot\zeta}\right|\,d\zeta
\le
\int_{|\zeta|>A}e^{-\frac{4}{\pi^2}|\zeta|^{2s}}\,d\zeta.
\]
Also
\[
\sup_{|\eta|\le R}\int_{|\zeta|>A}
e^{-|\zeta|^{2s}}\,d\zeta
=
\int_{|\zeta|>A}e^{-|\zeta|^{2s}}\,d\zeta.
\]
Both right-hand sides tend to \(0\) as \(A\to\infty\).

\medskip
\noindent
\textbf{(iii) convergence on bounded \(\zeta\)-balls.}
Fix \(A>0\). For \(t\) large, \(B_A\subset[-\pi/\lambda,\pi/\lambda]^d\). On \(B_A\), by Step 1 we have uniform  pointwise convergence, and by Step 2 an \(L^1(B_A)\)-dominating function \(e^{-\frac{4}{\pi^2}|\zeta|^{2s}}\). Hence dominated convergence yields
\[
\sup_{|\eta|\le R}
\left|
\int_{|\zeta|\le A}
\Big(
e^{-t\omega(\lambda\zeta)^s}e^{i\lambda x_t(\eta)\cdot\zeta}
-
e^{-|\zeta|^{2s}}e^{i\eta\cdot\zeta}
\Big)\,d\zeta
\right|
\longrightarrow0.
\]
Split \(I_t-\Phi_s(\eta)\) into \(|\zeta|\le A\) and \(|\zeta|>A\). First let \(t\to\infty\) (Step 3), then \(A\to\infty\) (Step 2). We obtain
\[
\sup_{|\eta|\le R}|I_t(\eta)-\Phi_s(\eta)|\to0.
\]
This is exactly \eqref{eq:lclt-uniform-compact}.
\end{proof}

We now prove that the first-order decay exponent obtained above cannot be improved in general, by testing on a single-lattice-shift datum.

\begin{proof}[\bf Proof of Proposition \ref{prop:optimality-first-order}.]
Set
\[
H_t(x):=G_t^{(s)}(x-e_1)-G_t^{(s)}(x).
\]
We only need a lower bound of order
\(
t^{-\frac d{2s}(1-\frac1p)}t^{-1/(2s)}.
\)

\smallskip
Let \(\Phi_s\) be the Euclidean \(2s\)-stable profile:
\[
\Phi_s(\eta)=\frac1{(2\pi)^d}\int_{\mathbb R^d}e^{-|\zeta|^{2s}}e^{i\eta\cdot\zeta}\,d\zeta.
\]
By Proposition~\ref{lem:lattice-stable-scaling}, for any \(R>0\), any lattice selection
\(x_t(\eta)\in\mathbb Z^d\) with
\(
\sup_\eta |x_t(\eta)-t^{1/(2s)}\eta|\le C_0
\),
we have
\[
\sup_{|\eta|\le R}
\left|
t^{d/(2s)}G_t^{(s)}\!\big(x_t(\eta)\big)-\Phi_s(\eta)
\right|\to0.
\]
Applying the same lemma to the shifted index \(x_t(\eta)-e_1\), we also get
\[
\sup_{|\eta|\le R}
\left|
t^{d/(2s)}G_t^{(s)}\!\big(x_t(\eta)-e_1\big)-\Phi_s(\eta)
\right|\to0,
\]
because
\[
\sup_{|\eta|\le R}
\left|
\big(x_t(\eta)-e_1\big)-t^{1/(2s)}
\Big(\eta-t^{-1/(2s)}e_1\Big)
\right|
\le C_0.
\]
Therefore,
\[
\sup_{|\eta|\le R}
\left|
t^{d/(2s)}H_t\!\big(x_t(\eta)\big)
-
\Big[
\Phi_s\!\big(\eta-t^{-1/(2s)}e_1\big)-\Phi_s(\eta)
\Big]
\right|\to0.
\]
Set \(h:=t^{-1/(2s)}\). Since \(\Phi_s\in C^2(\mathbb R^d)\),  for fixed \(R>0\),
\[
\sup_{|\eta|\le R}
\left|
\Phi_s(\eta-h e_1)-\Phi_s(\eta)+h\,\partial_1\Phi_s(\eta)
\right|
\le
\frac{h^2}{2}\,
\sup_{|\zeta|\le R+1}|\partial_{11}\Phi_s(\zeta)|
=O(h^2)=o(h),
\]
as \(h\to0\). Equivalently,
\[
\sup_{|\eta|\le R}
\left|
\frac{\Phi_s(\eta-h e_1)-\Phi_s(\eta)}{h}
+\partial_1\Phi_s(\eta)
\right|
\to0.
\]
Hence, we get
\[
\Phi_s\!\big(\eta-t^{-1/(2s)}e_1\big)-\Phi_s(\eta)
=
-\,t^{-1/(2s)}\partial_1\Phi_s(\eta)+o\,\big(t^{-1/(2s)}\big)
\]
uniformly for \(|\eta|\le R\). Hence
\begin{equation}
\label{eq:Ht-local-asymptotic}
\sup_{|\eta|\le R}
\left|
t^{\frac{d+1}{2s}}H_t\!\big(x_t(\eta)\big)+\partial_1\Phi_s(\eta)
\right|
\to0.
\end{equation}

\smallskip
Choose \(R>0\) so that \(\partial_1\Phi_s\not\equiv0\) on \(B_R\). Then
\[
A_{p,R}:=
\begin{cases}
\left(\displaystyle\int_{B_R}|\partial_1\Phi_s(\eta)|^p\,d\eta\right)^{1/p}>0,
&1\le p<\infty,\\[1ex]
\displaystyle\sup_{|\eta|\le R}|\partial_1\Phi_s(\eta)|>0,
&p=\infty.
\end{cases}
\]

\noindent
\textbf{(i) \(1\le p<\infty\).}
From \eqref{eq:Ht-local-asymptotic}, for \(t\) large,
\[
\left(\sum_{x_t(\eta)\in \Lambda_{t,R}}
\Big|t^{\frac{d+1}{2s}}H_t(x_t(\eta))\Big|^p\right)^{1/p}
\ge \frac12
\left(\sum_{x_t(\eta)\in \Lambda_{t,R}}
|\partial_1\Phi_s(\eta_x)|^p\right)^{1/p},
\]
where \(\Lambda_{t,R}:=\{x\in\mathbb Z^d:\ |x|\le Rt^{1/(2s)}\}\), and \(\eta_x:=t^{-1/(2s)}x\).
By Riemann-sum convergence,
\[
t^{-d/(2s)}\sum_{|x|\le Rt^{1/(2s)}}|\partial_1\Phi_s(t^{-1/(2s)}x)|^p
\to \int_{B_R}|\partial_1\Phi_s(\eta)|^p\,d\eta.
\]
Thus, for \(t\) large,
\[
\sum_{|x|\le Rt^{1/(2s)}}|H_t(x)|^p
\ge c\,t^{-p\frac{d+1}{2s}}\,t^{d/(2s)},
\]
and hence
\[
\|H_t\|_{\ell^p}
\ge c\,t^{-\frac{d+1}{2s}}\,t^{\frac d{2sp}}
=
c\,t^{-\frac d{2s}(1-\frac1p)}\,t^{-1/(2s)}.
\]

\noindent
\textbf{(ii) \(p=\infty\).}
From \eqref{eq:Ht-local-asymptotic},
\[
\sup_{|x|\le Rt^{1/(2s)}} t^{\frac{d+1}{2s}}|H_t(x)|
\to \sup_{|\eta|\le R}|\partial_1\Phi_s(\eta)|=A_{\infty,R}>0,
\]
so for \(t\) large,
\[
\|H_t\|_{\ell^\infty}\ge c\,t^{-(d+1)/(2s)}
= c\,t^{-d/(2s)}t^{-1/(2s)}.
\]

Combining both cases gives
\[
\|H_t\|_{\ell^p}
\ge c\,t^{-\frac d{2s}(1-\frac1p)}\,t^{-1/(2s)}
\]
for all \(t\ge t_0\). Multiplying by \(t^{\frac d{2s}(1-1/p)}\) finishes the proof.
\end{proof}

\subsubsection{Asymptotics under General \texorpdfstring{$\ell^1$}{l1} Class}
We next turn to initial data in the full \(\ell^1(\mathbb Z^d)\) class (without finite first moment), and establish qualitative large-time asymptotics with no universal quantitative convergence rate.

\begin{proof}[\bf Proof of Theorem \ref{thm:lattice-asymptotic-no-moment}.]
We use approximation plus the quantitative estimate in
Theorem~\ref{thm:lattice-first-order}.

\medskip
\noindent
\textbf{(i) \(\ell^1\)-convergence.}
Fix \(\varepsilon>0\).
Choose \(v_0\in \ell^1(\mathbb Z^d)\) with finite support such that
\begin{equation}
\label{eq:approx-u0-v0}
\|u_0-v_0\|_{\ell^1}<\varepsilon.
\end{equation}
Set
\[
v(t):=G_t^{(s)}*v_0,\qquad
M_v:=\sum_{x\in\mathbb Z^d}v_0(x).
\]
By \(\ell^1\)-contraction of the semigroup,
\begin{equation}
\label{eq:l1-contraction-u-v}
\|u(t)-v(t)\|_{\ell^1}
=
\|G_t^{(s)}*(u_0-v_0)\|_{\ell^1}
\le \|u_0-v_0\|_{\ell^1}
<\varepsilon,
\quad\forall\, t>0.
\end{equation}
Also,
\begin{equation}
\label{eq:mass-diff}
|M-M_v|
\le \|u_0-v_0\|_{\ell^1}
<\varepsilon.
\end{equation}
Hence
\[
\|u(t)-M G_t^{(s)}\|_{\ell^1}
\le
\|u(t)-v(t)\|_{\ell^1}
+\|v(t)-M_v G_t^{(s)}\|_{\ell^1}
+|M_v-M|\,\|G_t^{(s)}\|_{\ell^1}.
\]
Using \eqref{eq:l1-contraction-u-v}, \eqref{eq:mass-diff}, and (\ref{mass}),
\begin{equation}
\label{eq:l1-three-term}
\|u(t)-M G_t^{(s)}\|_{\ell^1}
\le
2\varepsilon+\|v(t)-M_v G_t^{(s)}\|_{\ell^1}.
\end{equation}
Since \(v_0\) has finite first moment, Theorem~\ref{thm:lattice-first-order} gives
\[
\|v(t)-M_v G_t^{(s)}\|_{\ell^1}
\le C\,\mathcal N_1(v_0)\,t^{-1/(2s)}\xrightarrow[]{t\to\infty}0.
\]
Taking \(\limsup_{t\to\infty}\) in \eqref{eq:l1-three-term},
\[
\limsup_{t\to\infty}\|u(t)-M G_t^{(s)}\|_{\ell^1}
\le 2\varepsilon.
\]
Since \(\varepsilon>0\) is arbitrary, \eqref{eq:lattice-L1-conv-no-moment} follows.

\medskip
\noindent
\textbf{(ii) \(\ell^\infty\)-convergence at the sharp scaling.}
Let
\[
R(t):=u(t)-M G_t^{(s)}.
\]
Using the semigroup property \(G_t^{(s)}=G_{t/2}^{(s)}*G_{t/2}^{(s)}\),
\[
R(t)=G_{t/2}^{(s)}*R(t/2).
\]
Therefore,
\[
\|R(t)\|_{\ell^\infty}
\le
\|G_{t/2}^{(s)}\|_{\ell^\infty}\,\|R(t/2)\|_{\ell^1}.
\]
Multiplying by \(t^{d/(2s)}\),
\begin{equation}
\label{eq:linfty-from-l1}
t^{\frac d{2s}}\|R(t)\|_{\ell^\infty}
\le
\Big(t^{\frac d{2s}}\|G_{t/2}^{(s)}\|_{\ell^\infty}\Big)\,
\|R(t/2)\|_{\ell^1}.
\end{equation}
By the kernel bound \(\|G_\tau^{(s)}\|_{\ell^\infty}\le C\,\tau^{-d/(2s)}\),
\[
t^{\frac d{2s}}\|G_{t/2}^{(s)}\|_{\ell^\infty}
\le C_{d,s}.
\]
Hence from \eqref{eq:linfty-from-l1},
\[
t^{\frac d{2s}}\|R(t)\|_{\ell^\infty}
\le C\,\|R(t/2)\|_{\ell^1}\xrightarrow[]{t\to\infty}0
\]
by \eqref{eq:lattice-L1-conv-no-moment}. This proves
\eqref{eq:lattice-Linfty-conv-no-moment}.

\medskip
\noindent
\textbf{(iii) \(\ell^p\)-convergence.}
For \(1\le p\le\infty\), interpolation gives
\[
\|R(t)\|_{\ell^p}
\le
\|R(t)\|_{\ell^1}^{1/p}\,
\|R(t)\|_{\ell^\infty}^{1-1/p}.
\]
Multiply by \(t^{\frac d{2s}(1-1/p)}\):
\[
t^{\frac d{2s}(1-1/p)}\|R(t)\|_{\ell^p}
\le
\|R(t)\|_{\ell^1}^{1/p}\,
\Big(t^{\frac d{2s}}\|R(t)\|_{\ell^\infty}\Big)^{1-1/p}.
\]
Both factors on the right tend to \(0\) by
\eqref{eq:lattice-L1-conv-no-moment} and \eqref{eq:lattice-Linfty-conv-no-moment},
hence \eqref{eq:lattice-Lp-conv-no-moment} follows.
\end{proof}

Here and in what follows, for \(z=(z_1,\dots,z_d)\in\mathbb R^d\), we denote
\[
\lfloor z\rfloor := (\lfloor z_1\rfloor,\dots,\lfloor z_d\rfloor)\in\mathbb Z^d.
\]
In particular,
\[
\big|\lfloor z\rfloor-z\big|\le \sqrt d .
\]

The next proposition shows that, in the mere \(\ell^1\)-framework, one cannot expect a uniform quantitative decay rate after the natural scaling \(t^{d/(2s)}\).  
More precisely, no prescribed vanishing profile \(\phi(t)\downarrow0\) can serve as a universal upper bound for
\(\,t^{d/(2s)}\|u(t)-M G_t^{(s)}\|_{\ell^\infty}\,\) over all unit-mass nonnegative data in \(\ell^1\).  
Equivalently, by a suitable choice of \(u_0\), the convergence to the asymptotic kernel can be made arbitrarily slow along a subsequence of times.

\begin{proof}[\bf Proof of Proposition \ref{prop:no-universal-rate-lattice}.]
Let \(\delta\in(0,1)\), and choose a summable positive sequence \((m_k)_{k\ge1}\) with
\[
\sum_{k=1}^\infty m_k=\delta.
\]
We construct points \(x_k\in\mathbb Z^d\), \(|x_k|\to\infty\), and times \(t_k\to\infty\) inductively, then set
\[
u_0:=(1-\delta)\delta_0+\sum_{k=1}^\infty m_k\,\delta_{-x_k}.
\]
Clearly \(u_0\in\ell^1\), \(u_0\ge0\), and \(\sum_x u_0(x)=1\).

Let
\[
u(t,0)=\sum_{y\in\mathbb Z^d}u_0(y)\,G_t^{(s)}(-y)
=(1-\delta)G_t^{(s)}(0)+\sum_{k=1}^\infty m_k\,G_t^{(s)}(x_k).
\]
Hence
\[
u(t,0)-G_t^{(s)}(0)
=
\sum_{k=1}^\infty m_k\bigl(G_t^{(s)}(x_k)-G_t^{(s)}(0)\bigr)
=
-\sum_{k=1}^\infty m_k\bigl(G_t^{(s)}(0)-G_t^{(s)}(x_k)\bigr).\]
Therefore
\begin{equation}\label{eq:error-at-origin}
    t^{\frac d{2s}}
\|u(t)-G_t^{(s)}\|_{\ell^\infty}
\ge
t^{\frac d{2s}}|u(t,0)-G_t^{(s)}(0)|
=
\sum_{k=1}^\infty m_k\,
\Bigl[t^{\frac d{2s}}\bigl(G_t^{(s)}(0)-G_t^{(s)}(x_k)\bigr)\Bigr].
\end{equation}

\medskip
\noindent
\textbf{(i) a fixed positive gap for one lump.}
By Proposition~\ref{lem:lattice-stable-scaling},
\[
t^{\frac d{2s}}G_t^{(s)}(0)\to \Phi_s(0)>0.
\]
By the Riemann--Lebesgue Lemma
\[
\Phi_s(\eta)\to0\quad\text{as }|\eta|\to\infty,
\]
we have \(\Phi_s\in C_0(\mathbb R^d)\). Also
\[
\Phi_s(0)=\frac1{(2\pi)^{d/2}}\int_{\mathbb R^d}e^{-|\xi|^{2s}}\,d\xi>0.
\]
Therefore, with \(\varepsilon:=\frac14\Phi_s(0)\), there exists \(R>0\) such that
\[
|\eta|\ge R\ \Longrightarrow\ \Phi_s(\eta)\le \varepsilon.
\]
Taking \(\rho_*:=R\), we get
\[
\Phi_s(\rho_*e_1)\le \frac14\Phi_s(0).
\]

Again by Lemma~\ref{lem:lattice-stable-scaling}, uniformly for \(|\eta|\le 2\rho_*\),
\[
t^{\frac d{2s}}G_t^{(s)}\!\big(\lfloor t^{1/(2s)}\eta\rfloor\big)\to \Phi_s(\eta).
\]
Hence there exists \(T_*\ge1\) such that for all \(t\ge T_*\):
\[
t^{\frac d{2s}}G_t^{(s)}(0)\ge \frac34\Phi_s(0),
\qquad
t^{\frac d{2s}}G_t^{(s)}\!\big(\lfloor t^{1/(2s)}\rho_* e_1\rfloor\big)\le \frac13\Phi_s(0).
\]
Therefore, for all \(t\ge T_*\),
\begin{equation}
\label{eq:uniform-gap}
t^{\frac d{2s}}
\Bigl(G_t^{(s)}(0)-G_t^{(s)}\!\big(\lfloor t^{1/(2s)}\rho_* e_1\rfloor\big)\Bigr)
\ge c_*,
\quad
c_*:=\frac{5}{12}\Phi_s(0)>0.
\end{equation}

\medskip
\noindent
\textbf{(ii) inductive choice of \(t_k,x_k\).}
Assume \(t_1,\dots,t_{k-1}\) and \(x_1,\dots,x_{k-1}\) are chosen.
Pick \(t_k\ge \max\{T_*,t_{k-1}+1\}\) so large that
\begin{equation}
\label{eq:choose-tk}
k\,\phi(t_k)\le \frac{c_*}{2}\,m_k.
\end{equation}
Now set
\[
x_k:=\left\lfloor t_k^{1/(2s)}\rho_* e_1\right\rfloor\in\mathbb Z^d.
\]
Then \(|x_k|\to\infty\), and by \eqref{eq:uniform-gap},
\[
t_k^{\frac d{2s}}
\bigl(G_{t_k}^{(s)}(0)-G_{t_k}^{(s)}(x_k)\bigr)\ge c_*.
\]
Hence from \eqref{eq:error-at-origin},
\[
t_k^{\frac d{2s}}
\|u(t_k)-G_{t_k}^{(s)}\|_{\ell^\infty}
\ge
m_k\,c_*
\ge
k\,\phi(t_k)
\]
by \eqref{eq:choose-tk}. This is exactly \eqref{eq:no-universal-rate-lattice}.
\end{proof}

\begin{remark}
\label{rem:conflict-no-rate-vs-first-moment}
In Proposition~\ref{prop:no-universal-rate-lattice}, the data are constructed as
\[
u_0=(1-\delta)\delta_0+\sum_{k\ge1}m_k\,\delta_{-x_k},
\qquad
x_k=\Big\lfloor t_k^{1/(2s)}\rho_*e_1\Big\rfloor .
\]
Hence
\[
\mathcal N_1(u_0):=\sum_{x\in\mathbb Z^d}|x|u_0(x)
=\sum_{k\ge1}m_k|x_k|
\sim
\sum_{k\ge1} m_k\,t_k^{1/(2s)}.
\]
On the other hand, to force
\[
t_k^{\frac d{2s}}
\|u(t_k)-G_{t_k}^{(s)}\|_{\ell^\infty}
\ge k\,\phi(t_k),
\]
the construction imposes
\[
m_k\gtrsim k\,\phi(t_k).
\]
Therefore
\[
m_k|x_k|
\gtrsim
k\,\phi(t_k)\,t_k^{1/(2s)}.
\]
If \(\phi\) is not faster than \(t^{-1/(2s)}\) (e.g. \(\phi(t)=t^{-1/(2s)}\), or slower),
the right-hand side is typically non-summable, so \(\sum_k m_k|x_k|=\infty\).  
This explains why the above ``no-universal-rate'' mechanism generally leaves the class
\(\{\mathcal N_1(u_0)<\infty\}\).

This is fully consistent with Theorem~\ref{thm:lattice-first-order}: if
\(\mathcal N_1(u_0)<\infty\), then
\[
t^{\frac d{2s}}
\|u(t)-M G_t^{(s)}\|_{\ell^\infty}
\le C(u_0)\,t^{-1/(2s)}.
\]
Hence one cannot have, for every \(\phi(t)\to0\), a lower bound of the form
\[
t_k^{\frac d{2s}}
\|u(t_k)-G_{t_k}^{(s)}\|_{\ell^\infty}
\ge k\,\phi(t_k)
\]
within the finite-first-moment class.
\end{remark}

The next corollary shows that the first-moment assumption in Theorem~\ref{thm:lattice-first-order}
is not only sufficient for a quantitative rate, but also essential for obtaining any uniform
rate on the whole \(\ell^1\) class with fixed mass.

\begin{proof}[\bf Proof of Corollary \ref{thm:lattice-asymptotic-no-moment}.]
If \eqref{eq:no-uniform-rate-unit-mass} were true, set \(\phi(t):=C\psi(t)\).
By Proposition~\ref{prop:no-universal-rate-lattice}, there exist
\(u_0\in\ell^1(\mathbb Z^d)\), \(u_0\ge0\), \(\sum_x u_0(x)=1\), and \(t_k\to\infty\)
such that
\[
t_k^{\frac d{2s}}
\|S_s(t_k)u_0-G_{t_k}^{(s)}\|_{\ell^\infty}
\ge k\,\phi(t_k)=k\,C\,\psi(t_k).
\]
But \eqref{eq:no-uniform-rate-unit-mass} gives the upper bound
\[
t_k^{\frac d{2s}}
\|S_s(t_k)u_0-G_{t_k}^{(s)}\|_{\ell^\infty}
\le C\,\psi(t_k),
\]
a contradiction for \(k\ge2\). Hence no such \(C\) exists.
\end{proof}

\section{Fractional Dirichlet Diffusion on Finite Subgraphs}
\label{sec:finite-subgraph-dirichlet}

In this section, we formulate and analyze the long-time behavior of fractional diffusion on a finite subgraph with homogeneous Dirichlet boundary condition. In contrast with the whole-lattice case, the finite-domain Dirichlet problem is governed by discrete spectrum and exhibits exponential decay toward the first eigenvalue.
\subsection{Setting and Equation}
\label{Setting and Equation}

Let $G = (V, E, \mu, w)$ is a weighted, connected graph, where $\mu$ is a positive measure, and $w : V \times V \to [0,\infty)$ is a symmetric edge function satisfying $w_{xy} = w_{yx},w_{xx} = 0$ and 
\[\sum_{y\in V}w_{xy}<\infty \quad \text{for all}\:\:x\in V.\]

Fix a nonempty finite connected subset \(\Omega\subset V\). Denote
\[
\partial\Omega:=\{x\in V\setminus\Omega:\exists\,y\in\Omega,\ x\sim y\},
\qquad
\overline\Omega:=\Omega\cup\partial\Omega.
\]

For \(u:\overline\Omega\to\mathbb R\), define the weighted graph Laplacian on \(\Omega\) by
\[
(\Delta_\Omega u)(x):=\frac1{\mu(x)}\sum_{y\in V}w_{xy}\bigl(u(y)-u(x)\bigr),
\qquad x\in\Omega,
\]
together with the Dirichlet exterior condition
\[
u(x)=0,\qquad x\in V\setminus\Omega.
\]
Equivalently, define the positive Dirichlet operator
\[
L_{\Omega,1}^D:=-\Delta_\Omega
\quad\text{on }\ell^2(\Omega,m).
\]

\medskip
\noindent
Given \(u:\Omega\to\mathbb R\), let \(\widetilde u\) be its zero extension:
\[
\widetilde u(x):=
\begin{cases}
u(x),&x\in\Omega,\\
0,&x\notin\Omega.
\end{cases}
\]
For \(s\in(0,1)\), define
\[
L_{\Omega,s}^D u
:=(-\Delta)^s_{\Omega}u
:=\Big((-\Delta)^s\widetilde u\Big)\Big|_{\Omega},
\qquad s\in(0,1),
\]
where \((-\Delta)^s\) is the fractional Laplacian on the whole graph
(see, e.g., \cite{chen2025logarithmic} for the corresponding nonlocal framework on general graphs).
Thus, for \(s\in(0,1]\), we treat
\[
L_{\Omega,s}^D=
\begin{cases}
(-\Delta)^s_{\Omega},& s\in(0,1),\\
-\Delta_\Omega,& s=1,
\end{cases}
\]
as a unified Dirichlet diffusion operator on \(\ell^2(\Omega,m)\).

We consider
\begin{equation}
\begin{cases}
\partial_t u(t,x)+L_{\Omega,s}^D u(t,x)=0,& t>0,\ x\in\Omega,\\
u(t,x)=0,& t>0,\ x\in V\setminus\Omega,\\
u(0,x)=u_0(x),& x\in\Omega.
\end{cases}
\end{equation}

Since \(\Omega\) is finite, for each \(s\in(0,1]\), \(L_{\Omega,s}^D\) is a symmetric positive
\(|\Omega|\times|\Omega|\) matrix on \(\ell^2(\Omega,m)\). Hence there exists an orthonormal eigenbasis
\(\{(\mu_{k,s},\psi_{k,s})\}_{k=1}^{|\Omega|}\) such that
\[
0<\mu_{1,s}<\mu_{2,s}\le\cdots\le\mu_{|\Omega|,s},
\qquad
L_{\Omega,s}^D\psi_{k,s}=\mu_{k,s}\psi_{k,s}.
\]

For \(u_0\in\ell^2(\Omega,m)\), the unique solution is
\[
u(t,\cdot)=e^{-tL_{\Omega,s}^D}u_0
=
\sum_{k=1}^{|\Omega|}
e^{-\mu_{k,s}t}\,
\langle u_0,\psi_{k,s}\rangle_{\ell^2(\Omega,m)}\,\psi_{k,s}.
\]
Therefore, as \(t\to\infty\), the asymptotic behavior is governed by
the principal mode \(e^{-\mu_{1,s}t}\psi_{1,s}\).

Moreover, since \(\Omega\) is finite, all spaces \(\ell^p(\Omega,m)\) (\(1\le p\le\infty\)) coincide as vector spaces (namely \(\mathbb{R}^{|\Omega|}\)); only the norms differ, and they are mutually equivalent.

     \subsection{Large-time Asymptotic Convergence }
We now turn to the long-time regime for the Dirichlet problem on a finite connected subgraph \(\Omega\).  
In contrast with the whole-lattice case \(\mathbb Z^d\), where large-time behavior is governed by  polynomial decay, the finite-domain Dirichlet dynamics is purely spectral: the semigroup \(e^{-tL_{\Omega,s}^D}\) is generated by a positive self-adjoint matrix with discrete eigenvalues.  Hence the first eigenvalue/eigenfunction pair provides the leading profile, while all higher modes decay exponentially faster.  
The goal of this subsection is to quantify this first-mode asymptotic expansion in \(\ell^p(\Omega,m)\), with explicit exponential rate given by the spectral gap.

\begin{proof}[\bf Proof of Theorem \ref{eq:finite-subgraph-renormalized-limit}.]
By the spectral representation of \(L_{\Omega,s}^D\),
\[
u(t)=\sum_{k=1}^{|\Omega|}
e^{-\mu_{k,s}t}\,
\langle u_0,\psi_{k,s}\rangle_{\ell^2(\Omega,m)}\,\psi_{k,s}.
\]
Subtract the principal mode:
\[
R(t):=
u(t)-e^{-\mu_{1,s}t}\langle u_0,\psi_{1,s}\rangle_{\ell^2(\Omega,m)}\,\psi_{1,s}
=
\sum_{k=2}^{|\Omega|}
e^{-\mu_{k,s}t}\,
\langle u_0,\psi_{k,s}\rangle_{\ell^2(\Omega,m)}\,\psi_{k,s}.
\]
Using orthonormality in \(\ell^2(\Omega,m)\),
\[
\|R(t)\|_{\ell^2(\Omega,m)}^2
=
\sum_{k=2}^{|\Omega|}
e^{-2\mu_{k,s}t}\,
\big|\langle u_0,\psi_{k,s}\rangle_{\ell^2(\Omega,m)}\big|^2
\le
e^{-2\mu_{2,s}t}
\sum_{k=2}^{|\Omega|}
\big|\langle u_0,\psi_{k,s}\rangle_{\ell^2(\Omega,m)}\big|^2
\le
e^{-2\mu_{2,s}t}\|u_0\|_{\ell^2(\Omega,m)}^2.
\]
Hence
\begin{equation}
\label{eq:R-L2-consistent}
\|R(t)\|_{\ell^2(\Omega,m)}
\le
e^{-\mu_{2,s}t}\|u_0\|_{\ell^2(\Omega,m)}.
\end{equation}

Since \(\Omega\) is finite, all \(\ell^p(\Omega,m)\)-norms are equivalent. Therefore, for each
\(1\le p\le\infty\), there exists \(C_{p,\Omega}>0\) such that
\[
\|f\|_{\ell^p(\Omega,m)}\le C_{p,\Omega}\|f\|_{\ell^2(\Omega,m)},
\qquad \forall f:\Omega\to\mathbb R.
\]
Applying this to \(f=R(t)\) and using \eqref{eq:R-L2-consistent}, we obtain
\[
\|R(t)\|_{\ell^p(\Omega,m)}
\le
C_{p,\Omega}e^{-\mu_{2,s}t}\|u_0\|_{\ell^2(\Omega,m)},
\]
which is exactly \eqref{eq:finite-subgraph-main-est}.

Multiplying \eqref{eq:finite-subgraph-main-est} by \(e^{\mu_{1,s}t}\) gives
\[
\left\|
e^{\mu_{1,s} t}u(t)-\langle u_0,\psi_{1,s}\rangle_{\ell^2(\Omega,m)}\,\psi_{1,s}
\right\|_{\ell^p(\Omega,m)}
\le
C_{p,\Omega}e^{-(\mu_{2,s}-\mu_{1,s})t}\|u_0\|_{\ell^2(\Omega,m)},
\]
namely \eqref{eq:finite-subgraph-renormalized-limit}. The convergence follows immediately.
\end{proof}

\begin{coro}
\label{cor:orthogonal-first-mode}
Under the assumptions of Theorem~\ref{thm:finite-subgraph-first-mode}, assume
\[
\langle u_0,\psi_{1,s}\rangle_{\ell^2(\Omega,m)}=0.
\]
Then, for every \(1\le p\le\infty\),
\[
\|u(t)\|_{\ell^p(\Omega,m)}
\le
C_{p,\Omega}\,e^{-\mu_{2,s}t}\,\|u_0\|_{\ell^2(\Omega,m)},
\qquad t\ge0.
\]
\end{coro}

\begin{proof}
If \(\langle u_0,\psi_{1,s}\rangle_{\ell^2(\Omega,m)}=0\), the principal mode vanishes in the spectral expansion:
\[
u(t)=\sum_{k=2}^{|\Omega|}
e^{-\mu_{k,s}t}\,
\langle u_0,\psi_{k,s}\rangle_{\ell^2(\Omega,m)}\,\psi_{k,s}.
\]
Equivalently,
\[
u(t)-e^{-\mu_{1,s}t}\langle u_0,\psi_{1,s}\rangle_{\ell^2(\Omega,m)}\psi_{1,s}
=u(t).
\]
Applying \eqref{eq:finite-subgraph-main-est} from
Theorem~\ref{thm:finite-subgraph-first-mode} gives
\[
\|u(t)\|_{\ell^p(\Omega,m)}
\le
C_{p,\Omega}\,e^{-\mu_{2,s}t}\,\|u_0\|_{\ell^2(\Omega,m)}.
\]
\end{proof}

\begin{remark}
\label{rem:finite-vs-lattice}
The large-time mechanism here is fundamentally different from that on \(\mathbb Z^d\). On \(\mathbb Z^d\), the profile is governed by the heat kernel \(G_t^{(s)}\), and decay is
typically polynomial in time, e.g. first-order corrections of size
\(t^{-1/(2s)}\). By contrast, on a finite domain with Dirichlet condition, the generator \(L_{\Omega,s}^D\)
has discrete spectrum with \(\mu_{1,s}>0\), and the solution is asymptotically
one-dimensional:
\[
u(t)\sim e^{-\mu_{1,s}t}\langle u_0,\psi_{1,s}\rangle_{\ell^2(\Omega,m)}\,\psi_{1,s},
\]
with exponential convergence rate
\[
e^{-(\mu_{2,s}-\mu_{1,s})t}
\]
after renormalization by \(e^{\mu_{1,s}t}\), where \(\mu_{2,s}-\mu_{1,s}\) is the spectral gap.
\end{remark}


\section{Positivity Improving of the Semigroup}
\label{sec:positivity-semigroup-comparison}

In this section, we clarify the positivity type of the semigroups,
distinguishing between the whole-lattice problem and the finite-domain Dirichlet problem.

\begin{proof}[\bf Proof of Proposition \ref{prop:zd-positivity-improving}.]
For \(s=1\), \(G_t^{(1)}\) is the continuous-time random-walk heat kernel on \(\mathbb Z^d\),
hence strictly positive at every site for every \(t>0\).

For \(0<s<1\), use subordination:
\[
e^{-t(-\Delta)^s}
=
\int_0^\infty e^{-r(-\Delta)}\,\eta_{t,s}(r)\,dr,
\]
where \(\eta_{t,s}(r)\ge0\), \(\int_0^\infty \eta_{t,s}(r)\,dr=1\), and \(\eta_{t,s}\not\equiv0\).
Therefore
\[
G_t^{(s)}(x)=\int_0^\infty G_r^{(1)}(x)\,\eta_{t,s}(r)\,dr.
\]
Since \(G_r^{(1)}(x)>0\) for all \(r>0\), \(x\in\mathbb Z^d\), the integral is strictly positive.
Then for \(u_0\ge0,\ u_0\not\equiv0\),
\[
(S_s(t)u_0)(x)=\sum_{y\in\mathbb Z^d}G_t^{(s)}(x-y)u_0(y)>0
\]
because at least one \(y\) has \(u_0(y)>0\), and each kernel factor is positive.
\end{proof}

Let \(\Omega\subset V\) be finite and connected, \(s\in(0,1]\), and
\(S_{\Omega,s}(t):=e^{-tL_{\Omega,s}^D}\).

\begin{proof}[\bf Proof of Proposition \ref{prop:finite-dirichlet-positivity-improving}.]
Write \(L_{\Omega,s}^D\) in matrix form:
\[
(L_{\Omega,s}^D u)(x)=\sum_{y\in\Omega} a_{xy}u(y),\qquad x\in\Omega.
\]
For both \(s\in(0,1)\) and \(s=1\), one has
\[
a_{xy}\le 0,\ \ x\neq y.
\]
Hence, $B:=-L_{\Omega,s}^D$ is a Metzler matrix. We now check irreducibility of \(B\).

\smallskip
\noindent
\textbf{(i) Case \(0<s<1\).}
By \cite{chen2025logarithmic}, we have
\begin{align*}
(-\Delta)^s u(x)
=\sum_{y\in V}W_s(x,y)\bigl(u(x)-u(y)\bigr),\quad
W_s(x,y)
=\frac{s}{\Gamma(1-s)}\int_0^\infty p(t,x,y)\,t^{-1-s}\,dt.
\end{align*}

Let \(\Omega\subset V\) be finite, and for \(u:\Omega\to\mathbb R\), define its zero extension
\[
\widetilde u(y)=
\begin{cases}
u(y),& y\in\Omega,\\
0,& y\in V\setminus\Omega.
\end{cases}
\]
Hence, for \(x\in\Omega\),
\begin{align*}
L_{\Omega,s}^D u(x)
&=\sum_{y\in V}W_s(x,y)\bigl(\widetilde u(x)-\widetilde u(y)\bigr)\\
&=\sum_{y\in\Omega}W_s(x,y)\bigl(u(x)-u(y)\bigr)
  +\sum_{y\in V\setminus\Omega}W_s(x,y)\,u(x).
\end{align*}
Write
\[
\kappa_s(x):=\sum_{y\in V\setminus\Omega}W_s(x,y)>0.
\]
Then
\[
L_{\Omega,s}^D u(x)
=\sum_{y\in\Omega,\ y\neq x}W_s(x,y)\bigl(u(x)-u(y)\bigr)+\kappa_s(x)u(x).
\]
Therefore,
\[
a_{xy}=
\begin{cases}
-\;W_s(x,y),& x\neq y,\ x,y\in\Omega,\\[1mm]
\displaystyle \sum_{z\in\Omega,\ z\neq x}W_s(x,z)+\kappa_s(x),& y=x.
\end{cases}
\]
Equivalently,
\[
a_{xx}=\sum_{z\in V,\ z\neq x}W_s(x,z),\qquad
a_{xy}=-W_s(x,y)<0\,\quad x\neq y,\ y\in\Omega.
\]
Thus,
\(B_{xy}>0\) for all \(x\neq y\). In particular, \(B\) is irreducible.

\smallskip
\noindent
\textbf{(ii) Case \(s=1\).}
Then \(L_{\Omega,1}^D\) is the Dirichlet graph Laplacian; \(B_{xy}>0\) iff \(x\sim y\) inside \(\Omega\).
Since \(\Omega\) is connected, the associated directed graph of \(B\) is strongly connected, hence \(B\)
is irreducible.

\smallskip
Therefore, in all cases \(B\) is an irreducible Metzler matrix. By \cite[Proposition 2.1]{Arora2025EventuallyPositive},
\[
e^{-tL_{\Omega,s}^D}=e^{tB}>0 \quad\text{entrywise for every } t>0.
\]
Hence, for \(u_0\ge0\), \(u_0\not\equiv0\), then
\[
\big(e^{-tL_{\Omega,s}^D}u_0\big)(x)
=\sum_{y\in\Omega}(e^{-tL_{\Omega,s}^D})_{xy}\,u_0(y)>0,
\quad \forall x\in\Omega,\ \forall t>0,
\]
which completes the proof.
\end{proof}


		
	 \bigskip\bigskip

\noindent {\bf  Conflicts of interest:} The authors declare that they have no conflicts of interest regarding this work.
\medskip

\noindent {\bf  Data availability:}  This paper has no associated data.\medskip

\noindent{{\bf Acknowledgements:}    
R. Chen is supported by China Scholarship Council,  Liujinxuan [2025] no. 37. Bo Li is supported by NNSF of China (12471094, 12571121), Zhejiang NSF of China (LMS26A010015), and the Qinshen Scholar Program of Jiaxing University.\\

	\printbibliography

	
	\noindent\textit{Rui Chen}: School of Mathematical Sciences, Fudan University,\\[1mm]
	Shanghai 200433,  China\\[1mm]
	Brandenburg University of Technology Cottbus--Senftenberg,\\[1mm]
	Cottbus 03046, Germany\\[1mm]
	\noindent\emph{Email:} \texttt{chenrui23@m.fudan.edu.cn}

            	\vspace{1em}
	\noindent\textit{Bo Li}: College of Data Science, Jiaxing University,\\[1mm]
	Jiaxing 314001,  China\\[1mm]
	\noindent\emph{Email:} \texttt{bli@zjxu.edu.cn}
	 
\end{document}